\documentclass[11pt]{amsart}

\usepackage{amsmath,amssymb}
\usepackage{indentfirst}
\usepackage[all]{xy}
\usepackage{color}
\newtheorem{theorem}{Theorem}[section]
\newtheorem{lemma}[theorem]{Lemma}
\newtheorem{corollary}[theorem]{Corollary}

\newtheorem{definition}[theorem]{Definition}

\newtheorem{proposition}[theorem]{Proposition}

\newcommand{\CW}{\operatorname{CW}}

\newcommand{\N}{\mathbb{N}}

\newcommand{\Z}{\operatorname{{\mathbb Z}}}

\renewcommand{\N}{\operatorname{{\mathbb N}}}
\newcommand{\Q}{\operatorname{{\mathbb Q}}}

\newcommand\dirlim{\operatorname{dirlim}}
\newcommand\ANR{\operatorname{ANR}}
\newcommand\PL{\operatorname{PL}}

\newcommand\inter{\operatorname{int}}

\newcommand\mesh{\operatorname{mesh}}
\newcommand\diam{\operatorname{diam}}
\newcommand\id{\operatorname{id}}
\newcommand\st{\operatorname{st}}

\newcommand\inv{^{-1}}

\begin{document}

\title{Alternate Proofs for the $n$-dimensional Resolution Theorems}
\author{Leonard R. Rubin}

\address{Department of Mathematics\\
University of Oklahoma\\
Norman, Oklahoma 73019\\
USA}\email{lrubin@ou.edu}

\author{Vera Toni\' c}

\address{Department of Mathematics\\
University of Rijeka\\
51000 Rijeka\\ Croatia}\email{vera.tonic@math.uniri.hr}

\date  {1 August 2021}

\begin{abstract}We present new, unified proofs for the cell-like-, $\Z/p$-,
and $\Q$-resolution theorems.  Our arguments employ extensions that are much
simpler than those used by our predecessors.  The techniques allow us to solve problems
involving cohomology groups by converting them into problems about homology groups.
We provide a coordinated general topological method for constructing the maps needed to
witness the resolution theorems simultaneously.
\end{abstract}



\maketitle 

\markboth{L. Rubin and V. Toni\' c}{Alternate Proofs for the Resolution Theorems}


\section[Introduction]
{Introduction}\label{intro}

For each abelian group $G$ and compact metrizable space $X$, $\dim_G X$ will denote
the cohomological dimension of $X$ modulo $G$.  One can find definitions of and facts about the theory
of cohomological dimension in \cite{Dr} and \cite{Ku}.  However, the material in Section \ref{EMacs}
will be sufficient for our needs. Throughout this paper, map means continuous function, and 
$\N$ starts with 1.

We are going to provide alternate, simpler, and unified proofs of the cell-like-, $\Z/p$-, and $\Q$-resolution
theorems in dimension $n$ (\cite{Ed}, \cite{Wa}, \cite{Dr}, \cite{Le}):

\begin{theorem}\label{cellresol}  Let  $n\in\N$ and $X$ be a nonempty metrizable compactum
with $\dim_{\Z}X\leq n$.  Then there exists a metrizable compactum $Z$ with $\dim Z\leq n$
and a surjective\footnote{It should be remarked that cell-like sets are never empty--see
Definition \ref{deftrivsh} and the remark just before it;
so the term surjective here is redundant.} cell-like map $\pi:Z\to X$.
\end{theorem}

\begin{theorem}\label{Zmodpreol} Let  $n\in\N$, $p\in\N_{\geq2}$, and $X$ be a nonempty metrizable compactum
with $\dim_{\Z/p}X\leq n$.  Then there exists a metrizable compactum $Z$ with $\dim Z\leq n$
and a surjective $\Z/p$-acyclic map $\pi:Z\to X$.
\end{theorem}

\begin{theorem}\label{rationalres} Let  $n\in\N_{\geq2}$ and $X$ be a nonempty metrizable compactum
with $\dim_{\Q}X\leq n$.  Then there exists a metrizable compactum $Z$ with $\dim Z\leq n$
and a surjective $\Q$-acyclic map $\pi:Z\to X$.
\end{theorem}

All terms with which the reader is not familiar will be defined in Section \ref{acyclicity}.
Theorem \ref{cellresol} was stated in \cite{Ed} and proved in \cite{Wa}.  Theorems
\ref{Zmodpreol} and \ref{rationalres} along with proofs can be found in
\cite{Dr} and \cite{Le} respectively.  All three of the resolution theorems are true in case $n=0$.
This holds because for any abelian group $G$ and nonempty metrizable compactum $X$,
$\dim_G X=0$ if and only if $\dim X=0$, and all singleton spaces are both cell-like and
$G$-acyclic (see Section \ref{acyclicity}). So one may take $Z=X$ and $\pi=\id_X$.
Theorem \ref{cellresol} with $n=1$ is true for a similar reason because, as stated in \cite{Wa},
$\dim_{\Z}X=1$ if and only if $\dim X=1$. Theorem \ref{rationalres} fails to be true
for $n=1$ as remarked in \cite{Le}.

One feature of all three of the previous proofs of the resolution theorems
is that they rely on a version of a result due to R. Edwards (see Theorem 4.2 of \cite{Wa})
even if somewhat hidden inside the arguments.
Because of this dependence, in each case the proofs employed relatively complicated
``extensions,'' e.g., Edwards-Walsh resolutions as in \cite{Wa} and \cite{Dr} (see also \cite{KY}),
and even more intricate ones in \cite{Le}.  The extensions were associated with triangulated compact polyhedra
that appeared in an inverse sequence representing the target space $X$.

Our main results yield unified proofs of Theorems \ref{cellresol},
\ref{Zmodpreol}, and \ref{rationalres} that do not rely on forms of Edwards' result.  As a consequence,
although we will use extensions (Sections \ref{fundext}, \ref{lemforext}, \ref{theextensions}, and \ref{propfor}),
these will be quite mild in comparison with those that were employed by our predecessors.
Furthermore, we are going to organize the previously disparate techniques
by providing a series of results that involve inverse sequences of compact
polyhedra (Section \ref{acyclicity}), the application
of basic notions of algebraic topology that connect homology with cohomology
(Sections \ref{acyclicity} and \ref{stars}), the theory of triangulated finite
polyhedra (Sections \ref{fundext} and \ref{lemforext}),
the connection between the theory of cohomological dimension and
Eilenberg-MacLane $\CW$-complexes (Section \ref{EMacs}),
and point-set topology (Sections \ref{ExtenLem}, \ref{Technicallemma}, and \ref{Recursion}).
All spaces that require a metric will be embedded in the Hilbert cube $I^\infty$
and their metrics will be induced from a fixed metric
$\rho$ on $I^\infty$ (Section \ref{ExtenLem}).   Section \ref{proofofmain} contains the
final steps of our proof of the resolution theorems.

\section[Acyclicity]
{Acyclicity}\label{acyclicity}

In this section we will provide the needed definitions of
``acyclicity'' including the concept of being ``cell-like.''
We will also have lemmas showing how to detect these properties
in the settings that will occur later.  The term {\it continuum} will refer
to a compact, connected Hausdorff space.  The notion of a space having the {\it shape of a point}
can be found on page 45 of \cite{MS}.  It can be seen from the remarks in the
first paragraph of page 45 of \cite{MS} that
whenever a space has the shape of a point, it cannot be empty and must be connected.

\begin{definition}\label{deftrivsh}A compact Hausdorff space is said to be
{\bf cell-like} if it has the shape of a point.
\end{definition}

It follows from this that each cell-like space must be a nonempty continuum.
In the original definition of the term cell-like, which was introduced by C. Lacher
(consult \cite{Ru}), it was required that such a space be metrizable.  The current loosening
of the requirement of metrizability is for convenience since this is essential for
understanding the main result, the cell-like resolution theorem for compact Hausdorff spaces, in
\cite{MR}; it has no impact on the work herein. Other equivalent definitions of cell-likenes can be found, for example, in \cite{Dr}.

\begin{definition}\label{defofGacyc}
Let $G$ be an abelian group and $X$ a continuum. One
says that $X$ is $G$-{\bf acyclic} if for all $k\in\N$,
${\Check H}^k(X;G)=0$.
\end{definition}

Note that if one wants to include $k=0$ in Definition \ref{defofGacyc}, then \v Cech cohomology should be reduced.
All singleton spaces are cell-like and are $G$-acyclic for any abelian group $G$.

\begin{definition}\label{trivshamap}Let $\pi:Z\to X$ be a
map and $G$ an abelian group.  It is said that $\pi$ is:
\begin{enumerate}
\item  {\bf cell-like} if each of its fibers is cell-like;
\item $G$-{\bf acyclic} if each of its fibers is $G$-acyclic.
\end{enumerate}
\end{definition}

Some authors require that in (1) of the preceding, the map $\pi$ be proper.  Since all the
spaces to which we apply Definition \ref{trivshamap} will be compact and metrizable,
then the maps we produce will be proper anyway.  So we will lose nothing by deleting
the properness requirement.  The maps $\pi$ that we shall design to satisfy
the resolution theorems will be surjective by construction.

\begin{lemma}\label{trivshape}Let $\mathbf{X}=(X_j,g_j^{j+1})$
be an inverse sequence of nonempty metrizable compacta and
$X=\lim\mathbf{X}$.  Then $X$ is cell-like if for each
$j\in\N$, there exists $i>j$ such that $g_j^i:X_i\to X_j$ is
null homotopic.\qed
\end{lemma}

Let $G$ be an abelian group.  There are several ways to demonstrate
that a given metrizable continuum $X$ is $G$-acyclic. The continuity
of \v Cech cohomology shows that:

\begin{lemma}\label{contofChe}For any abelian group $G$ and
inverse sequence $\mathbf{D}=(D_j,g_j^{j+1})$ of nonempty
compact polyhedra with $D=\lim\mathbf{D}$, $$\check H^k(D;G)=
\dirlim(H^k(D_j;G),H^k(g_j^{j+1};G)),$$ where for each $j$,
$H^k(g_j^{j+1};G):H^k(D_j;G)\to H^k(D_{j+1};G)$ is the
homomorphism induced by $g_j^{j+1}:D_{j+1}\to D_j$.\qed
\end{lemma}

\begin{lemma}\label{contCech}Let $\mathbf{D}=(D_j,g_j^{j+1})$
be an inverse sequence of nonempty compact polyhedra,
$D=\lim\mathbf{D}$, $k\in\N$, and $G$ an abelian group.  Then
$\Check H^k(D;G)=0$ if for each $j$, there exists $i>j$ such
that $H^k(g_j^i;G): H^k(D_j;G)\to H^k(D_i;G)$ is the trivial
homomorphism. \qed
\end{lemma}

When we apply Lemma \ref{contCech}, there will be given
$n\in\N$, and we will face three cases.  The first, that $k>n$,
will be taken care of because it will be true in
the situations we cover that $\dim D\leq
n$, so $\check H^k(D;G)=0$ automatically.  The second will be
that $1\leq k<n$, and will be covered in Lemma
\ref{detectGacy1}. The case that $k=n$ will be managed using
Lemma \ref{detectGacy2}.  We note that for all
$1\leq k\leq n$, we are going to make an appeal only to homology
groups.  We lay the groundwork for this now.

\begin{lemma}\label{detectGacy1}Let $\mathbf{D}=(D_j,g_j^{j+1})$
be an inverse sequence of nonempty compact polyhedra,
$D=\lim\mathbf{D}$, and $k\in\N$.  Suppose that for all $j\in\N$,
\begin{enumerate}\item there exists $i>j$
such that $H_k(g_j^i;\Z):
H_k(D_i;\Z)\to H_k(D_j;\Z)$ is the trivial homomorphism, and\item
$H_{k-1}(D_j;\Z)$ is free-abelian.
\end{enumerate}
Then for any abelian group $G$, $\Check H^k(D;G)=0$.
\end{lemma}

\begin{proof}Fix an abelian group $G$.
Let $j\in\N$, and choose $i>j$ as in (1) of
the hypothesis.  Using (2) and an application of Theorem 52.3(b),
page 318 of \cite{Mu}, one has that for each $s\in\{i,j\}$
$$\mathrm{Ext}(H_{k-1}(D_s;\Z),G)=0.$$
This and the short exact
sequence of Theorem 53.1 of \cite{Mu} (page 320, universal
coefficient theorem for cohomology) yield that the natural
Kronecker homomorphisms (see page 276 of \cite{Mu})
$\kappa_s:H^k(D_s;G)\to\mathrm{Hom}(H_k(D_s;\Z),G)$,
$s\in\{i,j\}$ are isomorphisms.

Let $$\lambda:
\mathrm{Hom}(H_k(D_j;\Z),G)\to\mathrm{Hom}(H_k(D_i;\Z),G)$$ be the homomorphism induced by
$H_k(g_j^i;\Z):H_k(D_i;\Z)\to H_k(D_j;\Z)$. Then we have
a commutative diagram of homomorphisms,
\begin{displaymath}
\xymatrix{
\mathrm{Hom} (H_k(D_j;\Z),G) \ar[rr]^{\lambda}  && \mathrm{Hom}(H_k(D_i;\Z),G) \\
H^k(D_j;G)\ar[u]^{\kappa_j} \ar[rr]_{H^k(g_j^i;G)} && H^k(D_i;G) \ar[u]^{\kappa_i}
}
\end{displaymath}

Let $u\in H^k(D_j;G)$.  We claim that $H^k(g_j^i;G)(u)=0$.
Since $\kappa_i$ is an isomorphism, then the preceding is true
if and only if $\kappa_i\circ H^k(g_j^i;G)(u)=0$. By the
diagram, this is true if and only if $\lambda\circ\kappa_j(u)
=0$. But $\kappa_j$ is also an isomorphism, so we only have to
show that $\lambda$ is the trivial homomorphism. Let
$f\in\mathrm{Hom}(H_k(D_j;\Z),G)$. Then $f:H_k(D_j;\Z)\to G$ is
a homomorphism, and $\lambda(f)=f\circ H_k(g_j^i;\Z)$.  The
second function in this composition is a trivial homomorphism
and hence $\lambda(f)$ is the trivial element of
$\mathrm{Hom}(H_k(D_i;\Z),G)$. Apply Lemma \ref{contCech} to complete this proof.
\end{proof}

\begin{corollary}\label{skiparound}Let $\mathbf{D}=(D_j,g_j^{j+1})$
be an inverse sequence of nonempty compact polyhedra,
$D=\lim\mathbf{D}$, and $k\in\N$.  Suppose that for all $j\in\N$,
\begin{enumerate}\item there exists $i>j$
such that $H_k(g_j^i;\Z): H_k(D_i;\Z)\to H_k(D_j;\Z)$ is
the trivial homomorphism, and\item there exist infinitely
many $i$ such that $H_{k-1}(D_i;\Z)$ is free-abelian.
\end{enumerate}
Then for any abelian group $G$, $\Check H^k(D;G)=0$.
\end{corollary}

\begin{lemma}\label{detectGacy2}Let $\mathbf{D}=(D_j,g_j^{j+1})$
be an inverse sequence of nonempty compact polyhedra,
$D=\lim\mathbf{D}$ and $n\in\N$.
Suppose that for all $j\in\N$,
\begin{enumerate}\item there exists $i>j$
such that $H_n(g_j^i;\Z/p):H_n(D_i;\Z/p)\to H_n(D_j;\Z/p)$ is the
trivial homomorphism,
\item $H_{n-1}(D_j;\Z)$ is free-abelian, and
\item $H_{n}(D_j;\Z)$ is free-abelian of finite rank.
\end{enumerate}
Then $\Check H^n(D;\Z/p)=0$.
\end{lemma}

\begin{proof}Let $j\in\N$, and choose $i>j$ as in (1) of the hypothesis.
Using (2) and an application of Theorem 54.4(c), page 328 of
\cite{Mu}, one sees that the torsion product of
$H_{n-1}(D_s,\Z)$ with $\Z/p$ equals $0$ for $s\in\{i,j\}$.  Hence
according to Theorem 55.1 of \cite{Mu} (universal coefficient
theorem for homology), there are natural isomorphisms
$\lambda_s:H_n(D_s;\Z)\otimes \Z/p\to H_n(D_s;Z/p)$, $s\in\{i,j\}$.

\begin{displaymath}
\xymatrix{
H_n(D_j;\Z)\otimes \Z/p \ar[d]_{\lambda_j} &&& H_n(D_i;\Z)\otimes \Z/p \ar[lll]_{H_n(g_j^i;\Z)\otimes id_{\Z/p}}
 \ar[d]_{\lambda_i} \\
H_n(D_j;\Z/p) &&& H_n(D_i;\Z/p) \ar[lll]^{H_n(g_j^i;\Z/p)}
}
\end{displaymath}

Since $H_n(g_j^i;\Z/p):H_n(D_i;\Z/p)\to H_n(D_j;\Z/p)$ is
the trivial homomorphism by (1), it follows that $H_n(g_j^i;\Z)\otimes id_{\Z/p}$
from the diagram above is also the trivial homomorphism.
Let us see what this means for the homomorphism $H_n(g_j^i;\Z):H_n(D_i;\Z)\to H_n(D_j;\Z)$. According to (3), the
groups $H_n(D_i;\Z)$ and  $H_n(D_j;\Z)$ are both (isomorphic to) finite direct sums of copies of $\Z$,
say $H_n(D_i;\Z) \cong \oplus_1^r \Z$ and $H_n(D_j;\Z) \cong \oplus_1^m \Z$, for
$r,m \in \N$. Let us take $e_1, \ldots , e_r$, and
$\overline e_1, \ldots , \overline e_m$ to be generators of
$H_n(D_i;\Z)$ and $H_n(D_j;\Z)$, respectively. Then, for each $q\in \{1, \ldots , r\}$, we have
\[H_n(g_j^i;\Z)(e_q)=(\alpha_1\cdot\overline e_1,\dots,
\alpha_m\cdot\overline e_m) \ \ \in H_n(D_j;\Z),\]
and since $H_n(g_j^i;\Z)\otimes id_{\Z/p}$ is trivial, this implies
that each of $\alpha_j$ is an integer divisible by $p$.

Now the commutative diagram of the proof of Lemma \ref{detectGacy1}
is valid for $G=\Z/p$, $k=n$, and the homomorphisms $\kappa_s$, $s\in\{i,j\}$ are
isomorphisms just as there because that only relies on (2).

\begin{displaymath}
\xymatrix{
\mathrm{Hom} (H_n(D_j;\Z),\Z/p) \ar[rr]^{\lambda}  && \mathrm{Hom}(H_n(D_i;\Z),\Z/p) \\
H^n(D_j;\Z/p)\ar[u]^{\kappa_j} \ar[rr]_{H^n(g_j^i;\Z/p)} && H^n(D_i;\Z/p) \ar[u]^{\kappa_i}
}
\end{displaymath}

Hence we may perform the steps analogous to those in the proof of Lemma
\ref{detectGacy1}, as follows.

Let $u\in H^n(D_j;\Z/p)$.  We claim that $H^n(g_j^i;\Z/p)(u)=0$.
To show this, as in the proof of Lemma \ref{detectGacy1} we only need to show that
$\lambda$ is the trivial homomorphism. Let
$f\in\mathrm{Hom}(H_n(D_j;\Z),\Z/p)$. Then $f:H_n(D_j;\Z)\to \Z/p$ is
a homomorphism, and $\lambda(f)=f\circ H_n(g_j^i;\Z)$.  But the properties of
$H_n(g_j^i;\Z)$ guarantee that, when followed by a homomorphism landing in $\Z/p$,
this yields the trivial homomorphism. Hence $\lambda(f)$ is the trivial element of
$\mathrm{Hom}(H_n(D_i;\Z),\Z/p)$. Apply Lemma \ref{contCech} to complete this proof.
\end{proof}

\begin{corollary}\label{anotherskip}Let $\mathbf{D}=(D_j,g_j^{j+1})$
be an inverse sequence of nonempty compact polyhedra,
$D=\lim\mathbf{D}$, and $n\in\N$. Suppose
that for all $j\in\N$,
\begin{enumerate}\item there exists $i>j$
such that $H_n(g_j^i;\Z/p):H_n(D_i;\Z/p)\to H_n(D_j;\Z/p)$ is the
trivial homomorphism, and there is an infinite subset $J\subset\N$ such that
for all $j\in J$,
\item  $H_{n-1}(D_j;\Z)$ is free-abelian, and
\item $H_n(D_j;\Z)$ is free-abelian of finite rank.
\end{enumerate}
Then $\Check H^n(D;\Z/p)=0$.
\end{corollary}

In \cite{Na} (see also Theorem 30.1 of \cite{IR}), it was proved that $\dim\leq n$ is
preserved in the limit of an inverse sequence of metrizable spaces.  Let us collect
this fact and some others that we need about limits of inverse sequences of metrizable compacta.

\begin{theorem}\label{presdim}Let $\mathbf{T}=(T_j,g_j^{j+1})$ be an inverse
sequence of metrizable compacta and $Z=\lim\mathbf{T}$.  Then,
\begin{enumerate}

\item $Z$ is a metrizable compactum,

\item if for all $j$, $T_j$ is connected, $Z$ is connected,

\item if for all $j$, $T_j\neq\emptyset$, $Z\neq\emptyset$, and

\item  if $n\geq0$ and for all $j$, $\dim T_j\leq n$, $\dim Z\leq n$.
\end{enumerate}
\end{theorem}

\section[Simplicial Complexes, Vertex Stars]
{Simplicial Complexes, Vertex Stars}\label{stars}

For each finite simplicial complex $S$, we are going to write $|S|$ to denote its
polyhedron with the weak topology which in this case is compact and metrizable.
We shall be employing subdivisions and simplicial approximations, so let us
note a couple of important facts that can be found in Chapter 3 of \cite{Sp}.

\begin{lemma}\label{sipapprox}Let $S$ be a finite simplicial complex.  Then,
\begin{enumerate}\item if $R$ is a subdivision of $S$, then there is a simplicial
approximation $\varphi:|R|\to|S|$ of $\id_{|S|}$, and
\item if $P$ is a compact polyhedron and $f:P\to|S|$ is a map, then there is a
triangulation $T$ of $P$ such that for every subdivision $L$ of $T$, there is
a simplicial approximation $g:|L|\to|S|$ of the map $f$.
\end{enumerate}
\end{lemma}

\begin{lemma}\label{ktok+1}Let $S$ be a finite simplicial complex such that $|S|$ is
contractible.  Then for each $k\in\N$, the inclusion $|S^{(k)}|\hookrightarrow|S^{(k+1)}|$
is homotopic in $|S^{(k+1)}|$ to a constant map.
\end{lemma}

\begin{proof}Each contracting homotopy of $|S^{(k)}|\hookrightarrow|S|$ can be replaced by
a $\PL$-homotopy $H:|S^{(k)}|\times[0,1]\to|S^{(k+1)}|$ where $H_0$ equals the inclusion
$|S^{(k)}|\hookrightarrow|S^{(k+1)}|$ and $H_1$ is a constant map.
\end{proof}

Next is an elementary fact from homology theory.

\begin{lemma}\label{elemfact}Let $S$ be a finite simplicial complex, $k\in\N$,
and $G$ an abelian group.  Then the inclusion map $|S^{(k+1)}|\hookrightarrow|S|$
induces an isomorphism of $H_k(|S^{(k+1)}|;G)$ onto $H_k(|S|;G)$.\qed
\end{lemma}

\begin{lemma}\label{contractsub}Let $T$ be a finite simplicial
complex and $S$ a subcomplex of $T$ such that $|S|$ is contractible.  Then for
each $n\in\N$, $1\leq k<n$, and abelian group $G$, $H_k(|S^{(n)}|;G)=0$.
\end{lemma}

\begin{proof}Since $|S|$ is contractible, then $H_k(|S|;G)=0$. By Lemma \ref{elemfact},
$k\in\N$ implies that the inclusion $|S^{(k+1)}|\hookrightarrow|S|$ induces an isomorphism of
$H_k(|S^{(k+1)}|;G)$ onto $H_k(|S|;G)$. This shows that $H_k(|S^{(k+1)}|;G)=0$.
On the other hand, since $1\leq k<n$, Lemma \ref{elemfact} shows that the inclusion
$|S^{(k+1)}|\hookrightarrow|S^{(n)}|$
induces an isomorphism of $H_k(|S^{(k+1)}|;G)$ onto $H_k(|S^{(n)}|;G)$, so the latter equals $0$.
\end{proof}

Let $T$ be a finite simplicial complex and $v\in T^{(0)}$.
Then $\st(v,T)$ will denote the {\it open star} of $v$ in $T$
and $\overline\st(v,T)$ will denote the {\it closed star}
of $v$ in $T$.  Of course, $\st(v,T)$ is an open neighborhood of
$v$ in the compact polyhedron $|T|$, $\overline\st(v,T)$ is
a closed subset of the compact polyhedron $|T|$, and
$\st(v,T)\subset\overline\st(v,T)$.  Moreover,
each of $\st(v,T)$ and $\overline\st(v,T)$ is contractible,
and there is a unique subcomplex $S$ of $T$ such that $|S|=
\overline\st(v,T)$.  Lemma \ref{contractsub} leads to the next fact about closed stars.

\begin{corollary}\label{contrachomol}Let $T$ be a finite simplicial
complex, $v\in T^{(0)}$, and $S$ the unique subcomplex of $T$ that triangulates
$\overline\st(v,T)$.  Suppose that $T_0$ is a subdivision of $T$.  Then,
\begin{enumerate}\item there is a unique subcomplex $S_0$ of $T_0$ with $|S_0|=|S|$, and
\item for each $n\in\N$, $1\leq k<n$, and abelian group $G$,
$H_k(|S_0^{(n)}|;G)=0$.
\end{enumerate}
\end{corollary}

As usual, when $G$ is an abelian group, $n\geq0$, and $T$ is a simplicial
complex, then $Z_n(T;G)$ will denote the set of (simplicial)
$G$-cycles among the $G$-chains.\footnote{Some
orientation on $T$ is assumed.}

\begin{definition}\label{boundarycyc}Whenever $\sigma$ is an
$(n+1)$-simplex of a simplicial complex $T$, then by
$\partial\sigma$ we shall mean the $n$-dimensional simplicial $\Z$-cycle of
$T^{(n)}$, i.e., $\partial\sigma\in Z_n(T^{(n)};\Z)$.\footnote{Later we are going
to use $\partial\sigma$ to denote the combinatorial boundary of
$\sigma$ (its boundary as a manifold); the reader will have no difficulty distinguishing between
these two uses of the same notation.}  Subsequently, however, when we deal with
homomorphisms induced by maps, we shall also treat $\partial\sigma$ as a singular $\Z$-cycle.
\end{definition}

\begin{lemma}\label{knockout}Let $E$ be a metrizable compactum, $p\in\N_{\geq2}$, $n\in\N$,
and $g\in H_n(E;\Z/p)$.  Then for each $r\in\N$ having the property that $p$ divides $r$, we have
$r\cdot g=0\in H_n(E;\Z/p)$.\qed
\end{lemma}

\begin{definition}\label{starstruck2}Let $T$ be a finite simplicial
complex, $v$ a vertex of $T$, and $n\in\N$. Denote by $S$ the
subcomplex of $T$ that triangulates $\overline\st(v,T)$, by $\mathcal{E}_{v,T}$ the set of
$(n+1)$-simplexes of $T$ having $v$ as a vertex, and by
$m_{v,T}$ the cardinality of $\mathcal{E}_{v,T}$.
\end{definition}

\begin{lemma}\label{isfreeab}Taking the notation from
$\mathrm{Definition\,\,\ref{starstruck2}}$, one has that
the group $H_n(|S^{(n)}|;\Z)$ is free abelian of rank $m_{v,T}$.
Moreover, if $G$ is an abelian group, $m_{v,T}\geq1$, and we
list $\mathcal{E}_{v,T}=\{\sigma_i\,|\,1\leq i\leq m_{v,T}\}$,
then:
\begin{enumerate}
\item for each $g\in G$ and $1\leq i\leq m_{v,T}$,
$g\cdot\partial\sigma_i$ lies in $Z_n(|S^{(n)}|;G)$,
\item
for each $z\in Z_n(|S^{(n)}|;G)$, there exists a set
$\{g_i\,|\,1\leq i\leq m_{v,T}\}\subset G$ such that
$z=\sum_{i=1}^{m_{v,T}}g_i\cdot\partial\sigma_i$,
\item if
$p\in\N_{\geq2}$, $h\in Z_n(|S^{(n)}|;\Z/p)$, and $r\in\Z$ is
a multiple of $p$, then $r\cdot h$ is an $n$-dimensional
boundary $\Z/p$-cycle, i.e., it is homologous to $0$ in
$Z_n(|S^{(n)}|;\Z/p)$, and\item if $p\in\N_{\geq2}$, $h\in
H_n(|S^{(n)}|;\Z/p)$, and $r\in\N$ is a multiple of $p$, then
$r\cdot h=0$. \qed\end{enumerate}
\end{lemma}

\begin{lemma}\label{zeroout}Let $T$ and $S$ be finite simplicial complexes, $n\in\N$, and
$f:|T|\to|S|$ a map.  Suppose that $\sigma$ is an $(n+1)$-simplex of $T$.
\begin{enumerate}\item If $\rho\in S^{(n)}$ and $f(\partial\sigma)\subset\rho$,
then $f:\partial\sigma\to|S|$ is homotopic to a constant map, so
for any abelian group $G$ and $g\in G$, $Z_n(f;G)(g\cdot\partial\sigma)$ is homologous
to $0$ in $Z_n(|S|;G)$.
\item Suppose there is an $(n+1)$-simplex $\rho\in S$ such that $f(\partial\sigma)\subset\partial\rho$.
If $p\in\N_{\geq2}$ and $f:\partial\sigma\to\partial\rho$ has degree
a multiple of $p$, then for any $g\in\Z/p$, $Z_n(f;\Z/p)(g\cdot\partial\sigma)$ is homologous
to $0$ in $Z_n(|S|;\Z/p)$.
If $f:\partial\sigma\to\partial\rho$ has degree $0$, i.e., is null homotopic, then
for any abelian group $G$ and $g\in G$, $Z_n(f;G)(g\cdot\partial\sigma)$ is homologous
to $0$ in $Z_n(|S|;G)$.  \qed
\end{enumerate}
\end{lemma}

\section[Fundamental Extensions]
{Fundamental Extensions}\label{fundext}

Definition \ref{allstickons} provides the blueprint for the technique we are going to
use to form the extensions that we shall need in our proof of the resolution theorems.

\begin{definition}\label{allstickons}Let $n\in\N$ and $(K,\Sigma^n)$ be a pair such that $K$ is a $\CW$-complex
and $\Sigma^n\subset K$ is an embedded copy of $S^n$.  For each finite simplicial complex $L$
and each $\sigma\in L$ with $\dim\sigma=n+1$,
we attach $K$ to $|L|$ along $\partial\sigma$ via a homeomorphism
$h_\sigma$ of $\Sigma^n\subset K$ to $\partial\sigma$, and we denote the attached copy $K_\sigma$.
The $\CW$-complex so obtained will be denoted $F_{L,K,\Sigma^n}$,
that is, $$F_{L,K,\Sigma^n}=|L|\cup\bigcup\{K_\sigma\,|\,(\sigma\in L)\wedge(\dim\sigma=n+1)\}.$$

It is of course understood that for each such $\sigma$, $K_\sigma\cap|L|=\partial\sigma$
and that if $\tau\in L$ with $\dim\tau=n+1$, then $K_\sigma\cap K_\tau=\partial\sigma\cap\partial\tau$.
\end{definition}

For each $\CW$-complex $K$ and $n\geq0$,
$K^{(n)}$ will denote the $n$-skeleton of $K$.

\begin{lemma}\label{itsnplusone}For each $(n,L,K,\Sigma^n)$ as in
$\mathrm{Definition\,\ref{allstickons}}$, it is true that $F_{L,K,\Sigma^n}^{(n+1)}=
|L^{(n+1)}|\cup\bigcup\{K_\sigma^{(n+1)}
\,|\,(\sigma\in L)\wedge(\dim\sigma=n+1)\}$.\qed
\end{lemma}

\begin{definition}\label{corepart}For each $(n,L,K,\Sigma^n)$ as in
$\mathrm{Definition\,\ref{allstickons}}$, denote,
$$F_{L,K,\Sigma^n}^*=|L^{(n)}|\cup\bigcup\{K_\sigma^{(n+1)}\,|\,(\sigma\in L)\wedge(\dim\sigma=n+1)\}.$$
\end{definition}

\begin{lemma}\label{littlethicker}For each $(n,L,K,\Sigma^n)$ as in
$\mathrm{Definition\,\ref{allstickons}}$,  $$F_{L,K,\Sigma^n}^*\subset F_{L,K,\Sigma^n}^{(n+1)}.\,\,\qed$$
\end{lemma}

For each $(n,L,K,\Sigma^n)$ as in Definition \ref{allstickons}, there exists a retraction $r:F_{L,K,\Sigma^n}\to|L|$
such that $r(K_\sigma\setminus\partial\sigma)\subset\inter\sigma$ whenever
$\sigma\in L$ and $\dim\sigma=n+1$. Let us make this formal.

\begin{definition}\label{pushdown}For each $(n,L,K,\Sigma^n)$ as in
$\mathrm{Definition\,\ref{allstickons}}$, fix a retraction $r_{L,K,\Sigma^n}:F_{L,K,\Sigma^n}\to|L|$
such that $r_{L,K,\Sigma^n}(K_\sigma\setminus\partial\sigma)
\subset\inter\sigma$ whenever $\sigma\in L$ and $\dim\sigma=n+1$.
\end{definition}

\begin{lemma}\label{meetinboundary}For each $(n,L,K,\Sigma^n)$ as in
$\mathrm{Definition\,\ref{allstickons}}$, the map $r_{L,K,\Sigma^n}:F_{L,K,\Sigma^n}\to|L|$
has the following properties:
\begin{enumerate}\item  $r_{L,K,\Sigma^n}(F_{L,K,\Sigma^n}^*)\subset|L^{(n+1)}|$,
\item $r_{L,K,\Sigma^n}(x)=x$ for all $x\in|L^{(n)}|$,
\item if $\sigma\in L$ with $\dim\sigma=n+1$, then $r_{L,K,\Sigma^n}(K_\sigma)\subset\sigma$, and
$r_{L,K,\Sigma^n}(K_\sigma\setminus\partial\sigma)\subset\inter\sigma$, and
\item the restriction of $r_{L,K,\Sigma^n}$ to $F_{L,K,\Sigma^n}^{(n+1)}$ is a retraction of $F_{L,K,\Sigma^n}^{(n+1)}$ to
its subspace $|L^{(n+1)}|$.\qed
\end{enumerate}
\end{lemma}

Let $n\in\N$ and $L$ be a finite simplicial complex. There exists a
subdivision $S$ of $L$ having the property that for each $\sigma\in L$
with $\dim\sigma=n+1$, $|N(\partial\sigma,S)|$ is a regular neighborhood of $\partial\sigma$,
where of course $N(\partial\sigma,S)$ denotes the simplicial
neighborhood of $\partial\sigma$ in $S$.  For example, we may choose $S$ to be
the second barycentric subdivision of $L$.
\begin{lemma}\label{anANR}Let $n\in\N$, $L$ a finite simplicial complex,
and $S$ a subdivision of $L$ having the property that for each $\sigma\in L$
with $\dim\sigma=n+1$, $|N(\partial\sigma,S)|$ is a regular neighborhood of $\partial\sigma$.
Then for each $\sigma\in L$ with $\dim\sigma=n+1$,
\begin{enumerate}\item$\partial\sigma\subset|N(\partial\sigma,S)|$,
\item $\partial\sigma$ is a strong deformation retract of $|N(\partial\sigma,S)|$, and
\item $|N(\partial\sigma,S)|$ is a compact $\ANR$.\qed
\end{enumerate}
\end{lemma}

\begin{definition}\label{aregnbhd}Let $n\in\N$ and $L$ be a finite simplicial complex.
Let $S$ be a subdivision of $L$ having the property that for each $\sigma\in L$
with $\dim\sigma=n+1$, $|N(\partial\sigma,S)|$ is a regular neighborhood of $\partial\sigma$.
We shall refer to such an $S$ as an $n$-{\bf regular subdivision} of $L$.
\end{definition}

\begin{definition}\label{addskirts}Let $(n,L,K,\Sigma^n)$ be as in
$\mathrm{Definition\,\ref{allstickons}}$. Suppose that $S$ is an $n$-regular subdivision of $L$.
For each $\sigma\in L$ with $\dim\sigma=n+1$, let $$K_{\sigma,+}=K_\sigma\cup|N(\partial\sigma,S)|.$$
\end{definition}

Taking into account Lemma \ref{anANR}(2,3) and Definition \ref{addskirts}, one obtains the next fact.

\begin{lemma}\label{wearaskirt}Let $(n,L,K,\Sigma^n)$ be as in $\mathrm{Definition\,\ref{allstickons}}$.
Suppose that $S$ is an $n$-regular subdivision of $L$. Then for each $\sigma\in L$ with $\dim\sigma=n+1$,
\begin{enumerate}\item$K_{\sigma,+}$ is an absolute neighborhood extensor for compact metrizable spaces,
\item$K_\sigma$ is a strong deformation retract of $K_{\sigma,+}$, and
\item whenever $Y$ is a metrizable compactum with $Y\tau K$, then $Y\tau K_{\sigma,+}$.\qed
\end{enumerate}
\end{lemma}

\section[Lemma for an Extension]
{Lemma for an Extension}\label{lemforext}

Lemma \ref{Edwardstype} is critical for our proof of Proposition \ref{pushandshove2}.  It is one
of the key ingredients in our technique for avoiding complicated extensions, that is, to staying
away from Edwards type lemmas as we mentioned in the Introduction.

\begin{lemma}\label{Edwardstype}Let $(n,L,K,\Sigma^n)$ be as in $\mathrm{Definition\,\ref{allstickons}}$.
Suppose that $E$ is a compact polyhedron, $f:E\to|L|$ a map, $S$ an $n$-regular subdivision of $L$,
and that for each $\sigma\in L$ with $\dim\sigma=n+1$, the map $f|f\inv(|N(\partial\sigma,S)|):f\inv(|N(\partial\sigma,S)|)\to
|N(\partial\sigma,S)|\subset K_{\sigma,+}$ extends to a map of $E$ to $K_{\sigma,+}$. Let $N$ be a
triangulation of $E$ that admits a simplicial approximation $f_1:|N|\to|L|$ of the map $f$.
Then there exists a map $h:|N^{(n+1)}|\to F_{L,K,\Sigma^n}^*$ $(\mathrm{Definition\,\ref{corepart}})$ such that,
\begin{enumerate}\item$h(x)=f_1(x)$ for all $x\in f_1\inv(|L^{(n)}|)\cap|N^{(n+1)}|$, and
\item for each $\sigma\in L$ with $\dim\sigma=n+1$, $h(f_1\inv(\sigma)\cap|N^{(n+1)}|)\subset K_\sigma$.
\end{enumerate}
\end{lemma}

\begin{proof}Let $M$ be a subdivision of $N$ that admits a map $f_2:|M|\to|S|$ such that
$f_2$ is simultaneously a simplicial approximation of $f$ and $f_1$.  Temporarily fix
$\sigma\in L$ with $\dim\sigma=n+1$.  We claim that:

$(\dag_1)$  $f_2(f_1\inv(\partial\sigma))\subset\partial\sigma$,

$(\dag_2)$  $f_1\inv(\partial\sigma)\subset f_2\inv(\partial\sigma)$, and

$(\dag_3)$ $f_2|f_1\inv(\partial\sigma):f_1\inv(\partial\sigma)\to\partial\sigma$ is homotopic
to $f_1|f_1\inv(\partial\sigma):f_1\inv(\partial\sigma)\to\partial\sigma$.

Plainly $(\dag_2)$ follows from $(\dag_1)$; let us prove $(\dag_1)$.
Let $x\in f_1\inv(\partial\sigma)$; then $f_1(x)\in\partial\sigma$ which implies that there
exists $\tau\in S$ such that $f_1(x)\in\inter\tau\subset\tau\subset\partial\sigma$.  Since $f_2:|M|\to|S|$
is a simplicial approximation of $f_1$, then $f_2(x)\in\tau\subset\partial\sigma$ as needed
for $(\dag_1)$.  Since $\{f_1(x),f_2(x)\}\subset\tau\subset\partial\sigma$, then $f_1$ and $f_2$
are ``straight-line'' homotopic as maps to $\partial\sigma$, so $(\dag_3)$ is also established.

Now let us prove some parallel facts:

$(\dag_4)$  $f(f_2\inv(\partial\sigma))\subset|N(\partial\sigma,S)|$,

$(\dag_5)$  $f_2\inv(\partial\sigma)\subset f\inv(|N(\partial\sigma,S)|$, and

$(\dag_6)$  $f|f_2\inv(\partial\sigma):f_2\inv(\partial\sigma)\to|N(\partial\sigma,S)|$ is homotopic
to $f_2|f_2\inv(\partial\sigma):f_2\inv(\partial\sigma)\to\partial\sigma\subset|N(\partial\sigma,S)|$.

Plainly $(\dag_5)$ follows from $(\dag_4)$. We proceed to prove $(\dag_4)$.
Let $x\in f_2\inv(\partial\sigma)$; there exists $\tau\in S$ with $f(x)\in\inter\tau$.  Since
$f_2:|M|\to|S|$ is a simplicial approximation of $f$ and $\tau\in S$, then $f_2(x)\in\tau$.
But $f_2(x)\in\partial\sigma\cap\tau$ which implies that $\tau\subset|N(\partial\sigma,S)|$.
Hence $f(x)\in\inter\tau\subset\tau\subset|N(\partial\sigma,S)|$,
which gives us $(\dag_4)$.  Since $\{f(x),f_2(x)\}\subset\tau\subset|N(\partial\sigma,S)|$, then $f$
and $f_2$ are ``straight-line'' homotopic as maps to $|N(\partial\sigma,S)|$, so $(\dag_6)$ is also established.

By hypothesis, the map $$f|f\inv(|N(\partial\sigma,S)|):f\inv(|N(\partial\sigma,S)|)\to
|N(\partial\sigma,S)|\subset K_{\sigma,+}$$ extends to a map of $E$ to $K_{\sigma,+}$.
This, alongside $(\dag_5)$ and $(\dag_6)$, shows that,

$(\dag_7)$ $f_2|f_2\inv(\partial\sigma):f_2\inv(\partial\sigma)\to\partial\sigma\subset|N(\partial\sigma,S)|$
extends to a map of $E$ to $K_{\sigma,+}$.

Taking into consideration $(\dag_7)$, $(\dag_2)$, and $(\dag_3)$, one concludes that:

$(\dag_8)$ $f_1|f_1\inv(\partial\sigma):f_1\inv(\partial\sigma)\to\partial\sigma\subset|N(\partial\sigma,S)|$
extends to a map of $E$ to $K_{\sigma,+}$.

Making use of $(\dag_8)$ and the strong deformation retraction guaranteed us by Lemma \ref{anANR}(2),
we have,

$(\dag_9)$ $f_1|f_1\inv(\partial\sigma):f_1\inv(\partial\sigma)\to\partial\sigma\subset|N(\partial\sigma,S)|$
extends to a map $f_{1,\sigma}:E\to K_\sigma$.

Having accomplished the preceding for each fixed $\sigma$, now notice that
$|L^{(n+1)}|=|L^{(n)}|\cup\bigcup\{\sigma\,|\,(\sigma\in L)\wedge(\dim\sigma=n+1)\}$.
Of course, $f_1(|N^{(n+1)}|)\subset|L^{(n+1)}|$.  Hence,

$(\dag_{10})$  $|N^{(n+1)}|=\big(f_1\inv(|L^{(n)}|)\cup\bigcup\{f_1\inv(\sigma)\,|\,
(\sigma\in L)\wedge(\dim\sigma=n+1)\}\big)\cap|N^{(n+1)}|$.

There exists a subcomplex $N_0$ of $N^{(n+1)}$ such that $|N_0|=f_1\inv(|L^{(n)}|)\cap|N^{(n+1)}|$.
Moreover, for each $\sigma\in L$ with $\dim\sigma=n+1$, there is also a subcomplex $N_\sigma$ of
$N^{(n+1)}$ such that $|N_\sigma|=f_1\inv(\sigma)\cap|N^{(n+1)}|$.
Using this and $(\dag_{10})$ we see that,

$(\dag_{11})$  $|N^{(n+1)}|=|N_0|\cup\bigcup\{|N_\sigma|\,|\,
(\sigma\in L)\wedge(\dim\sigma=n+1)\}$.

Taking into account $(\dag_9)$, for each $\sigma\in L$ with $\dim\sigma=n+1$, we put

$(\dag_{12})$ $h_\sigma=f_{1,\sigma}\big||N_\sigma|:|N_\sigma|\to K_\sigma$.

As an application of $(\dag_9)$, we find that if $\sigma\in L$ and $\dim\sigma=n+1$, then

$(\dag_{13})$ for each $x\in f_1\inv(\partial\sigma)\cap|N^{(n+1)}|$, $f_{1,\sigma}(x)=f_1(x)$.

If $\tau\in L$ and $\dim\tau=n+1$, then

$(\dag_{14})$  $f_1\inv(\partial\sigma)\cap f_1\inv(\partial\tau)\cap
|N^{(n+1)}|=f_1\inv(\partial\sigma\cap\partial\tau)\cap
|N^{(n+1)}|\subset f_1\inv(|L^{(n)}|)\break\cap|N^{(n+1)}|=|N_0|$.

Let,

$(\dag_{15})$ $h_0=f_1\big||N_0|:|N_0|\to|L^{(n)}|$.

Making use of $(\dag_{11})$--$(\dag_{15})$ and $(\dag_9)$, we see that, $$h=h_0\cup
\bigcup\{h_\sigma\,|\,(\sigma\in L)\wedge(\dim\sigma=n+1)\}$$
is a map of $|N^{(n+1)}|$ to $F_{L,K,\Sigma^n}^*$ satisfying $(1)$ and $(2)$.  This completes our proof.
\end{proof}

\section[Eilenberg-MacLane Complexes for $\Z$, $\Z/p$, $\Q$]
{Eilenberg-MacLane Complexes for $\Z$, $\Z/p$, $\Q$}\label{EMacs}

When $X$ is a space, we shall write $X\tau K$ to mean that $X$
is an absolute co-extensor for $K$.  This simply means that if $A$ is a closed subset of $X$,
and $f:A\to K$ is a map, then there is a map $g:X\to K$ that extends $f$.  We assume that the
reader is familiar with the concept of an Eilenberg-MacLane $\CW$-complex $K$ of type $K(G,n)$; such
complexes exist for each abelian group $G$, and for a given abelian group $G$ any $\CW$-complexes $K_1$, $K_2$ of
type $K(G,n)$ are homotopy equivalent.  It is well-known that if $X$ is a metrizable compactum and $K_1$, $K_2$ are
Eilenberg-MacLane $\CW$-complexes of type $K(G,n)$, then $X\tau K_1$ if and only if $X\tau K_2$.
According to Theorem 1.1 of \cite{Dr},

\begin{lemma}\label{equivforcohdim} For each $n\in\N$,
compact metrizable space $X$, and Eilenberg-MacLane $\CW$-complex $K$ of type $K(G,n)$, $\dim_G X\leq n$ if and only
if $X\tau K$.
\end{lemma}

When we encounter a space $X$ which is homeomorphic to $S^n$, then we shall
assume that there is a fixed isomorphism between $\pi_n(X)$ and $\Z$.
In this way if $Y$ is also homeomorphic to $S^n$ and $f:X\to Y$ is a map,
then the degree of $f$, denoted $\deg(f)$, is well-defined.

In what follows, we will use standard constructions of the $\CW$-complexes $K(\Z,n)$, $K(\Z/p,n)$,
and $K(\Q,n)$ for $n\in\N$ and $p\geq2$. For $K(\Z,n)$ and $K(\Q,n)$ we shall require that $n\geq2$, but
for $K(\Z/p,n)$, $n\geq1$ will be permitted.  These restrictions on $n$ will typically be implicit in the sequel.
Each of $K(\Z,n)$ and $K(\Z/p,n)$ respectively contains a unique ``canonical'' copy $\Sigma_{\Z}=K(\Z,n)^{(n)}$ and
$\Sigma_{\Z/p}=K(\Z/p,n)^{(n)}$ of $S^n$. In this setting,
$K(\Z,n)^{(n+1)}=K(\Z,n)^{(n)}=\Sigma_{\Z}$; moreover
$K(\Z/p,n)^{(n+1)}$ is a standard Moore space of type $(\Z/p,n)$, i.e., an $(n+1)$-cell
$B^{n+1}$ attached to $\Sigma_{\Z/p}$ via a map of $\partial(B^{n+1})$ of degree $p$.
We shall denote $\kappa_{\Z}:\Sigma_{\Z}\hookrightarrow K(\Z,n)$ and
$\kappa_{\Z/p}:\Sigma_{\Z/p}\hookrightarrow K(\Z/p,n)$.

$(*)$  In case $f:S^n\to\Sigma_{\Z/p}$
is a map and $\pi_n(\kappa_{\Z/p})\circ\pi_n(f):\pi_n(S^n)\to\pi_n(K(\Z/p,n))$ is the trivial
map, then $\deg(f)$ is a multiple of $p$.

The situation with $K(\Q,n)$ is different because $K(\Q,n)^{(n)}$ is a countable wedge, say
$\bigvee\{S_i^n\,|\,i\in\N\}$ of copies $S_i^n$ of $S^n$. So there is not a canonically unique choice of
$S^n$ as we had above for the groups $\Z$ and $\Z/p$.

\begin{definition}\label{fixediso}Select a fixed isomorphism $\psi:\pi_n(K(\Q,n))\to\Q$. We may do this
so that for some $i\in\N$, if we define $\Sigma_{\Q}=S_i^n$, then the following is true.  Let
$\kappa_{\Q}:\Sigma_{\Q}\hookrightarrow K(\Q,n)$.  Then the composition $\psi\circ\pi_n(\kappa_{\Q}):\pi_n(\Sigma_{\Q})
\to\Q$ sends $\pi_n(\Sigma_{\Q})$ isomorphically onto the standard copy of $\Z$ in $\Q$.
\end{definition}

\begin{lemma}\label{twoofthegroups}Let $n\in\N$ and $\mu$ be an $(n+1)$-simplex.
\begin{enumerate}\item If $f:\partial\mu\to\Sigma_{\Z}$ is a map such that
$\kappa_{\Z}\circ f:\partial\mu\to K(\Z,n)$ extends to a map $f^*:\mu\to K(\Z,n)$, then
$f$ is homotopic to a constant map.
\item If $f:\partial\mu\to\Sigma_{\Z/p}$ is a map such that $\kappa_{\Z/p}\circ f:\partial\mu\to K(\Z/p,n)$
extends to a map $f^*:\mu\to K(\Z/p,n)^{(n+1)}$, then $\deg(f)$ is a multiple of $p$.
\item If $f:\partial\mu\to\Sigma_{\Q}$ is a map such
that $\kappa_{\Q}\circ f:\partial\mu\to K(\Q,n)$ extends to a map $f^*:\mu\to K(\Q,n)$, then
$\deg(f)=0$.
\end{enumerate}
\end{lemma}

\begin{proof}Note that $\pi_n(\kappa_{\Z}):\pi_n(\Sigma_{\Z})\to\pi_n(K(\Z,n))$ is an
isomorphism.  Assume that $f$ is not homotopic to a constant map.  It follows that there exists $k\in\N$ such
that $\deg(f)\in\{k,-k\}$.  Select a generator $g$ of $\pi_n(\partial\mu)$.  Then for some
generator $\overline g$ of $\pi_n(\Sigma_{\Z})$, $\pi_n(f)(g)=k\cdot\overline g$
is a nontrivial element of $\pi_n(\Sigma_{\Z})$.  Hence
$\pi_n(\kappa_{\Z})(k\cdot\overline g)=\pi_n(\kappa_{\Z})(\pi_n(f)(g))
=\pi_n(\kappa_{\Z}\circ f)(g)$ is a nontrivial element of
$\pi_n(K(\Z,n))$.  But by the assumption in (1), $\kappa_{\Z}\circ f:\partial\mu\to K(\Z,n)$ extends
to the map $f^*:\mu\to K(\Z,n)$ showing that $\pi_n(\kappa_{\Z}\circ f)(g)=\pi_n(\kappa_{\Z})
\circ\pi_n(f)(g)$ is the trivial element of $\pi_n(K(\Z,n))$.  Since $\pi_n(\kappa_{\Z})$
is an isomorphism, it has to be true that $\pi_n(f)(g)$ is the trivial element of
$\pi_n(K(\Z,n))$, a contradiction.  We have demonstrated (1).

Next we prove (2). Note that $K(\Z/p,n)^{(n+1)}$ is a Moore space of type $(\Z/p,n)$.  We have that
$\kappa_{\Z/p}\circ f:\partial\mu\to K(\Z/p,n)^{(n+1)}$ extends to a map $f^*:\mu\to K(\Z/p,n)^{(n+1)}$.
This shows that $\pi_n(\kappa_{\Z/p})\circ\pi_n(f)=\pi_n(\kappa_{\Z/p}\circ f):\pi_n(\partial\mu)\to
\pi_n(K(\Z/p,n)^{(n+1)})$ is the trivial map of $\pi_n(\partial\mu)$.  It follows from $(*)$ that
$\deg(f)$ is a multiple of $p$.

Let us prove (3).  Suppose that $\deg(f)\neq0$.
Consider the case that $\deg(f)\in\{1,-1\}$.  Then $\pi_n(f):\pi_n(\partial\mu)\to\pi_n(\Sigma_{\Q})$
is an isomorphism.  This and Definition \ref{fixediso}, show that

$(\dag_1)$  $\psi\circ\pi_n(\kappa_{\Q})\circ\pi_n(f)=\psi\circ\pi_n(\kappa_{\Q}\circ f)$
sends $\pi_n(\partial\mu)$ isomorphically onto $\Z\subset\Q$.

But the assumption that
$\kappa_{\Q}\circ f$ extends to a map of $\mu$ into $K(\Q,n)$ shows that the homomorphism
$\pi_n(\kappa_{\Q}\circ f)$ is trivial, which in turn implies that
$\psi\circ\pi_n(\kappa_{\Q}\circ f)$ sends $\pi_n(\partial\mu)$ trivially into $\Q$, a contradiction.

We consider the case that $k\geq2$ and $\deg(f)\in\{k,-k\}$.  Then

$(\dag_2)$ there is a generator $\overline g$ of $\pi_n(\Sigma_{\Q})$ such that
$\pi_n(f)(g)=k\cdot\overline g$ which is a nontrivial element of $\pi_n(\Sigma_{\Q})$.

Hence, $\psi\circ\pi_n(\kappa_{\Q})(k\cdot\overline g)=\psi\circ\pi_n(\kappa_{\Q})(\pi_n(f)(g))=
\psi\circ\pi_n(\kappa_{\Q})\circ\pi_n(f)(g)=\psi\circ\pi_n(\kappa_{\Q}\circ f)(g)$
is a nontrivial element of $\Z\subset\Q$.  But the assumption in (3) implies that $\kappa_{\Q}\circ f$
is homotopic to a constant map, so $\psi\circ\pi_n(\kappa_{\Q}\circ f)(g)$ is the trivial element of $\Z\subset\Q$,
a contradiction.
\end{proof}

\section[Simple Extensions]
{Simple Extensions}\label{theextensions}

In \cite{Wa}, \cite{Dr} (see also \cite{RT}), and \cite{Le}, the proofs of Theorems
\ref{cellresol}--\ref{rationalres} made
use of increasingly more complex extensions.  One of our goals, as stated in the Introduction, is to provide
proofs of these theorems that employ much simpler extensions.  The groundwork for producing them was laid in
Sections \ref{fundext} and \ref{EMacs}.  Now we shall provide the precise definitions for the
extensions we are going to use.

When $n\in\N$ is given, $G\in\{\Z,\Z/p,\Q\}$, and $L$ is a finite simplicial complex, then the $\CW$-complex
$F_{L,K,\Sigma^n}$ of Definition \ref{allstickons} will be constructed with $\Sigma^n=\Sigma_G$,
the particular canonical copy of $S^n$ in $K$ (see Definition \ref{fixediso} in case
$G=\Q$), and we shall conserve notation by simply denoting $F_{L,K,\Sigma^n}=F_{L,K}$.
For each $\sigma\in L$ with $\dim\sigma=n+1$, the inclusion $\kappa_\sigma:\partial\sigma
\hookrightarrow K_\sigma$ will correspond to the inclusion $\kappa_G:\Sigma_{G}
\hookrightarrow K$, $G\in\{\Z,\Z/p,\Q\}$.

Lemma \ref{twoofthegroups} can now be applied to arrive at the following.

\begin{lemma}\label{missingskeleton}Let $L$ be a finite simplicial complex, $\sigma\in L$ with
$\dim\sigma=n+1$, and $\mu$ an $(n+1)$-simplex.
\begin{enumerate}\item If $K=K(\Z,n)$ and $f:\partial\mu\to\partial\sigma$ is a map such that
$f=\kappa_\sigma\circ f:\partial\mu\to K_\sigma$ extends to a map $f^*:\mu\to K_\sigma$, then $\deg(f)=0$.
\item If $K=K(\Z/p,n)$ and $f:\partial\mu\to\partial\sigma$ is a map such that
$f=\kappa_\sigma\circ f:\partial\mu\to K_\sigma$ extends to a map $f^*:\mu\to K_\sigma^{(n+1)}$, then
$\deg(f)$ is a multiple of $p$.
\item If $K=K(\Q,n)$ and $f:\partial\mu\to\partial\sigma$ is a map such
that $f=\kappa_\sigma\circ f:\partial\mu\to K_\sigma$ extends to a map $f^*:\mu\to K_\sigma$, then
$\deg(f)=0$.
\end{enumerate}
\end{lemma}

\section[Proposition for Control of the $(n+1)$-skeleton]
{Proposition for Control of the $(n+1)$-skeleton}\label{propfor}

\begin{proposition}\label{pushandshove2}Let $n\in\N$, $K\in\{K(\Z,n),K(\Z/p,n),K(\Q,n)\}$, $E$ a
compact polyhedron, $L$ a finite simplicial complex,
$\overline p:E\to|L|$ a map, and $S$ an $n$-regular subdivision of $L$.  Suppose that for each $\sigma\in L$
with $\dim\sigma=n+1$, the map $\overline p|(\overline p)\inv(|N(\partial\sigma,S)|):
(\overline p)\inv(|N(\partial\sigma,S)|)\to
|N(\partial\sigma,S)|\subset K_{\sigma,+}$ extends to a map of $E$ to $K_{\sigma,+}$. Let $N$ be a
triangulation of $E$ that admits a simplicial approximation $\overline p_1:|N|\to|L|$ of the map $\overline p$,
$L_0$ a subdivision of $N$, and $\varphi:|L_0|\to|N|$ a simplicial approximation to $\id:|L_0|\to|N|$.
Then there exists a map $g_K:|N^{(n+1)}|\to|L^{(n+1)}|$ such that:

\begin{enumerate}
\item  $g_K$ is an $L$-modification of both $\overline p_1\big||N^{(n+1)}|:|N^{(n+1)}|\to
|L^{(n+1)}|\subset|L|$
and $\overline p\big||N^{(n+1)}|:|N^{(n+1)}|\to|L|$;

\item  $g_K(|N^{(n+1)}|)\subset|L^{(n+1)}|$, $g_K(|N^{(n)}|)\subset|L^{(n)}|$,
$g_K\circ\varphi(|L_0^{(n+1)}|)\subset|L^{(n+1)}|$, and $g_K\circ\varphi(|L_0^{(n)}|)\subset|L^{(n)}|$;

\item for each $x\in|L_0^{(n+1)}|$, there exists $\sigma\in L$ with $\{\overline p(x),
g_K\circ\varphi(x)\}\subset\sigma$;

\item for each $\nu\in L_0^{(n+1)}$, $\overline p_1(\varphi(\nu))\in L^{(n+1)}$,
and $g_K(\varphi(\nu))\subset\overline p_1(\varphi(\nu))$;

\item  $g_K(x)=\overline p_1(x)$ for all $x\in(\overline p_1)\inv(|L^{(n)}|)\cap|N^{(n+1)}|$;

\item in case $K=K(\Z,n)$ and $\tau\in N^{(n+1)}$, then
$g_K(\tau)\subset|L^{(n)}|$ so $g_K(|N^{(n+1)}|)\subset|L^{(n)}|$; the preceding implies that
$g_K\circ\varphi(|L_0^{(n+1)}|)\subset g_K(|N^{(n+1)}|)\subset|L^{(n)}|$.
\smallskip

Now let $v\in N^{(0)}$, $N_v$, $L_{0,v}$ respectively be the subcomplexes of $N$, $L_0$
with $|N_v|=|L_{0,v}|=\overline\st(v,N)$, and $\widetilde L$ be a subcomplex of $L$ with
$g_K(|N_v|)\subset\widetilde L$.  Then,
\smallskip

\item both $g_K(|N_v^{(n)}|)\subset|\widetilde L^{(n)}|$ and
$g_K(\varphi(|L_{0,v}^{(n)}|))\subset|\widetilde L^{(n)}|$,

\item In case $K=K(\Z,n)$, then $g_K\circ\varphi\big||L_{0,v}^{(n)}|:
|L_{0,v}^{(n)}|\to|\widetilde L^{(n)}|$ is null homotopic;

\item in case $K=K(\Z/p,n)$, then the induced homomorphism $H_n(g_K\circ\varphi
\big||L_{0,v}^{(n)}|;\Z/p):H_n(|L_{0,v}^{(n)}|;\Z/p)\to H_n(|\widetilde L^{(n)}|;\Z/p)$ is trivial, and

\item in case $K=K(\Q,n)$, then the induced homomorphism $H_n(g_K\circ\varphi
\big||L_{0,v}^{(n)}|;\Z):H_n(|L_{0,v}^{(n)}|;\Z)\to H_n(|\widetilde L^{(n)}|;\Z)$ is trivial.
\end{enumerate}
\end{proposition}

\begin{proof}We shall treat the pair $\big(|N^{(n+1)}|,(\overline p_1)\inv(|L^{(n)}|)\cap|N^{(n+1)}|\big)$
as a relative $\CW$-complex (see 7.6.4, p. 401 of \cite{Sp} for the notion of a relative $\CW$-complex).
We want to apply Lemma \ref{Edwardstype} to $(n,L,K,\Sigma^n)$ with $K$ as given and $\Sigma^n$
the canonical copy of $S^n$ in $K$ (see Section \ref{theextensions}).  The map $f$ of Lemma \ref{Edwardstype}
is replaced by our $\overline p:E\to|L|$; the simplicial approximation $f_1$ is the current $\overline p_1$.
Let $h_K:|N^{(n+1)}|\to F_{L,K}^*$ be a map as granted to us by Lemma \ref{Edwardstype}.
From Lemma \ref{Edwardstype}(1), it is seen that $h_K$ is a pair map,
$$h_K:(|N^{(n+1)}|,(\overline p_1)\inv(|L^{(n)}|)\cap|N^{(n+1)}|)\to(F_{L,K}^*,|L^{(n)}|).$$
Of course, $(F_{L,K}^*,|L^{(n)}|)$ is also a relative $\CW$-complex.
Part $(1)$ of Lemma \ref{Edwardstype} shows in addition that since $h_K$ equals the simplicial map
$\overline p_1$ on $(\overline p_1)\inv(|L^{(n)}|)\break\cap|N^{(n+1)}|$, then

$(\dag_1)$  $h_K$ is simplicial from $(\overline p_1)\inv(|L^{(n)}|)\cap|N^{(n+1)}|$ to $|L^{(n)}|$.

Using the cellular approximation theorem (7.17, p. 404 of \cite{Sp}),
one finds a relative cellular approximation $$h_K^*:(|N^{(n+1)}|,(\overline p_1)\inv(|L^{(n)}|)\cap|N^{(n+1)}|)
\to(F_{L,K}^*,|L^{(n)}|)$$ of the map $h_K$. This is done so that

$(\dag_2)$ $h_K^*=h_K=\overline p_1$ on $(\overline p_1)\inv(|L^{(n)}|)\cap|N^{(n+1)}|$.

The cellular approximation theorem requires that if $x\in|N^{(n+1)}|$ and $h_K(x)$ lies in a
subcomplex $M$ of $F_{L,K}^*$, then $h_K^*(x)\in M^{(n+1)}$.  It then follows that

$(\dag_3)$  if $\sigma\in L$ with $\dim\sigma=n+1$, $x\in|N^{(n+1)}|$, and
$h_K(x)\in K_\sigma$, then $h_K^*(x)\in K_\sigma^{(n+1)}$.

Using Lemma \ref{meetinboundary}(1,2), we may treat

$(\dag_4)$ $r_{L,K}:(F_{L,K}^*,|L^{(n)}|)\to(|L^{(n+1)}|,|L^{(n)}|)$ as a pair map where
$r_{L,K}(x)=x$ for all $x\in|L^{(n)}|$.

Now we define

$(\dag_5)$  $g_K:|N^{(n+1)}|\to|L^{(n+1)}|$ by $g_K(x)=r_{L,K}\circ h_K^*(x)$.

From this and $(\dag_4)$, it follows that,

$(\dag_6)$  if $x\in|N^{(n+1)}|$ and $h_K^*(x)\in|L^{(n)}|$, then $g_K(x)=h_K^*(x)$.

To check (1), let $x\in|N^{(n+1)}|$ and let $\tau\in L$ be the unique simplex with $\overline p(x)\in\inter\tau$.
Since $\overline p_1$ is a simplicial approximation of $\overline p$, then
$\overline p_1(x)\in\tau$. But $x\in|N^{(n+1)}|$ and $\overline p_1$
is simplicial, so there exists a unique $\sigma\in L^{(n+1)}$ with $\sigma\subset\tau$
and $\overline p_1(x)\in\inter\sigma$.

If $x\in(\overline p_1)\inv(|L^{(n)}|)\cap|N^{(n+1)}|$, then by $(\dag_2)$,
$h_K^*(x)=h_K(x)=\overline p_1(x)\in|L^{(n)}|$.
Using $(\dag_6)$, one sees that $g_K(x)=h_K^*(x)=\overline p_1(x)\in\sigma\subset\tau$.  So $g_K$ is an
$L$-modification of $\overline p_1$ with respect to the domain $(\overline p_1)\inv(|L^{(n)}|)\cap|N^{(n+1)}|$.

The alternate case is that $x\notin(\overline  p_1)\inv(|L^{(n)}|)$ and
hence $\overline p_1(x)\in|L^{(n+1)}|\setminus|L^{(n)}|$.
So $\dim\sigma=n+1$.  By Lemma \ref{Edwardstype}(2),
$h_K((\overline p_1)\inv(\sigma)\cap|N^{(n+1)}|)\subset K_\sigma$.
But $x\in(\overline p_1)\inv(\sigma)\cap|N^{(n+1)}|$, so $h_K(x)\in K_\sigma$.
This puts $h_K^*(x)\in K_\sigma$.
Applying Lemma \ref{meetinboundary}(2,3), one sees that $r_{L,K}(K_\sigma)\subset\sigma$.
Therefore $g_K(x)=r_{L,K}\circ h_K^*(x)\in\sigma\subset\tau$.  So $g_K$ is an
$L$-modification of $\overline p_1$ with respect to the complementary part of the domain.

Putting the two cases together, one sees that $g_K$ is an $L$-modification of
$\overline p_1\big||N^{(n+1)}|:|N^{(n+1)}|\to|L^{(n+1)}|$.
So the first part of (1) is true. To show that $g_K$ is an $L$-modification of
$\overline p\big||N^{(n+1)}|:|N^{(n+1)}|\to|L|$, let $x\in|N^{(n+1)}|$ and $\sigma\in L$
be the unique simplex with $\overline p(x)\in\inter\sigma$. Then because $\overline p_1:|N|\to|L|$
is a simplicial approximation of $\overline p$, one has that $\overline p_1(x)\in\sigma$.
By the first part of (1), $g_K(x)\in\sigma$, as required to complete our proof of (1).

Statement (2) follows from (1) and the fact that $\overline p_1\big||N^{(n+1)}|:|N^{(n+1)}|\to|L^{(n+1)}|$
is simplicial, and hence both $g_K(|N^{(n+1)}|)\subset|L^{(n+1)}|$ and $g_K(|N^{(n)}|)\break\subset|L^{(n)}|$
hold.

Now we prove (3).  Let $x\in|L_0^{(n+1)}|\subset|N|$, $\sigma\in L$ be the unique simplex
with $\overline p(x)\in\inter\sigma$,
$\mu\in L_0^{(n+1)}$ the unique simplex with $x\in\inter\mu$, and $\lambda\in N$
the unique simplex with $x\in\inter\lambda$.  Then $\varphi(\mu)\in N^{(n+1)}$
and $\varphi(x)\in\varphi(\mu)\subset\lambda$.
Now, $\overline p_1:|N|\to|L|$ is a simplicial approximation of $\overline p$,
so $\overline p_1(x)\in\sigma$.  It follows that $\overline p_1(\lambda)\subset\sigma$, so
$\overline p_1\circ\varphi(x)\in\sigma$. But, $g_K$ is an $L$-modification of $\overline p_1$,
so $g_K\circ\varphi(x)\in\sigma$.  Thus $\{g_K\circ\varphi(x),\overline p(x)\}\subset\sigma$.

The first part of (4) is true because both $\overline p_1$ and $\varphi$ are simplicial.
The second part comes from the first part and (1). One obtains (5) from  $(\dag_2)$, $(\dag_4)$, and $(\dag_5)$.

In order to prove properties from (6) onward, let us notice some additional facts.
Since $\overline p_1:|N|\to|L|$ is simplicial, then for each $\tau\in N^{(n+1)}$, $\overline p_1(\tau)\in
L^{(n+1)}$. Let us use this in proving that,

$(\dag_7)$ if $\tau\in N^{(n+1)}$ and $\dim\tau=n+1=\dim\overline p_1(\tau)$,
then $h_K^*(\tau)\subset K_{\overline p_1(\tau)}^{(n+1)}$,
and $h_K^*(\partial\tau)\subset K_{\overline p_1(\tau)}^{(n)}=\partial(\overline p_1(\tau))$.

Now $\tau\subset(\overline p_1)\inv(\overline p_1(\tau))\cap|N^{(n+1)}|$.
By Lemma \ref{Edwardstype}(2) with $\overline p_1(\tau)$ in place of $\sigma$,
$h_K((\overline p_1)\inv(\overline p_1(\tau))\cap|N^{(n+1)}|)\subset K_{\overline p_1(\tau)}$. By $(\dag_3)$,
$h_K^*((\overline p_1)\inv(\overline p_1(\tau))\cap|N^{(n+1)}|)\subset K_{\overline p_1(\tau)}^{(n+1)}$.
Hence the first part of $(\dag_7)$ holds true.  The second part is true because
$h_K^*$ is cellular. Next we show that,

$(\dag_8)$ if $\tau\in N^{(n+1)}$ and $\dim\tau=n+1=\dim\overline p_1(\tau)$,
then $g_K(\tau)\subset\overline p_1(\tau)$,
$g_K(\partial\tau)\subset\partial(\overline p_1(\tau))$, and $g_K|\partial\tau=h_K^*|\partial\tau$.

By the first part of $(\dag_7)$, $h_K^*(\tau)\subset K_{\overline p_1(\tau)}^{(n+1)}$, and
by $(\dag_5)$, $x\in\tau$ implies that $g_K(x)=r_{L,K}\circ h_k^*(x)$.  Apply this with
Lemma \ref{meetinboundary}(3) to obtain the first part of $(\dag_8)$. The second
and third parts follow respectively from the second part of $(\dag_7)$ and $(\dag_6)$.

If we put together $(\dag_7)$ and $(\dag_8)$, we arrive at the next fact.

$(\dag_9)$ if $\tau\in N^{(n+1)}$ and $\dim\tau=n+1=\dim(\overline p_1(\tau))$, then
$g_K|\partial\tau:\partial\tau\to\partial(\overline p_1(\tau))$ extends to a map
of $\tau$ to $K_{\overline p_1(\tau)}$.

To prove (6), consider first the case that $\overline p_1(\tau)\subset|L^{(n)}|$, i.e.,
$\overline p_1(\tau)\in L^{(n)}$.  Then use
(5) to see that $g_K(\tau)=\overline p_1(\tau)\subset|L^{(n)}|$.  In the complementary case, $\overline p_1(\tau)\in
L^{(n+1)}$ and $\dim\overline p_1(\tau)=n+1$.
Note that since $K=K(\Z,n)$, then $K_{\overline p_1(\tau)}^{(n+1)}=
K_{\overline p_1(\tau)}^{(n)}=\partial(\overline p_1(\tau))$.  This and $(\dag_7)$ show that $h_K^*(\tau)
\subset\partial(\overline p_1(\tau))\subset|L^{(n)}|$. Now employ $(\dag_6)$ to see that
$g_K(\tau)=h_K^*(\tau)\subset\partial(\overline p_1(\tau))\subset|L^{(n)}|$.  The first part of (7) follows
from the second part of (2); its second part uses the first part and the fact that $\varphi$ is a simplicial
approximation of the identity.  As to (8),
use Lemma \ref{ktok+1} to see that there is a homotopy $H:|L_{0,v}^{(n)}|\times[0,1]\to|L_{0,v}^{(n+1)}|$
between the inclusion $\iota:|L_{0,v}^{(n)}|\hookrightarrow|L_{0,v}^{(n+1)}|$ and
a constant map.  Now apply (6) to the homotopy $g_K\circ\varphi\circ H:|L_{0,v}^{(n)}|\times[0,1]
\to|\widetilde L^{(n)}|$ (with appropriate domain restrictions) to get a null homotopy
as desired in (8).

In proving (9) and (10) we shall use the following fact from Lemma \ref{isfreeab}.

$(\dag_{10})$
There exists $m\in\N$ and a set $\{\sigma_i\,|\,1\leq i\leq m\}$ of
$(n+1)$-simplexes of $N_v$, such that for any abelian group $G$ and each $z\in Z_n(|N_v^{(n)}|;G)$,
there is a set $\{g_i\,|\,1\leq i\leq m\}\subset G$ with $z=\sum_{i=1}^m g_i\cdot\partial\sigma_i$.

Statement (9) will be true if we can prove that

$(\dag_{11})$  $H_n(g_K;\Z/p):H_n(|N_v^{(n)}|;\Z/p)\to
H_n(\widetilde L^{(n)};\Z/p)$ is trivial.

We see that $(\dag_{11})$ is true if for each $z=\sum_{i=1}^m g_i\cdot\partial\sigma_i
\in Z_n(|N_v^{(n)}|;\Z/p)$, $Z_n(g_K;\Z/p)(z)$ is homologous to $0$ in
$Z_n(\widetilde L^{(n)};\Z/p)$. Let $1\leq i\leq m$ and with $\sigma_i$ in place of $\tau$,
let us consider the two cases that we explored in our proof of (6) above.  In the first
of these, we have that $\overline p_1(\sigma_i)\in L^{(n)}$ and $g_K(\sigma_i)=\overline p_1(\sigma_i)$.
It follows that $g_K$ maps $\partial\sigma_i$ trivially into $|\widetilde L^{(n)}|$.  So,
$Z_n(g_K;\Z/p)(g_i\cdot\partial\sigma_i)$ is homologous to $0$ in $Z_n(\widetilde L^{(n)};\Z/p)$.
In the second case, using $(\dag_9)$ with $\sigma_i$ replacing $\tau$, one has that $g_K|\partial\sigma_i:
\partial\sigma_i\to\partial(\overline p_1(\sigma_i))$ extends
to a map of $\sigma_i$ into $K_{\overline p_1(\sigma_i)}^{(n+1)}$.  By Lemma
\ref{missingskeleton}(2), $g_K|\partial\sigma_i:
\partial\sigma_i\to\partial(\overline p_1(\sigma_i))$ has degree a multiple of $p$.
From this and the first part of Lemma \ref{zeroout}(2), it follows that $Z_n(g_K;\Z/p)(g_i\cdot\partial\sigma_i)$
is homologous to $0$ in $Z_n(\partial(\overline p_1(\sigma_i));\Z/p)$ and hence in
$Z_n(\widetilde L^{(n)};\Z/p)$.  This gives us (9).  To obtain (10), use the same argument
we just made, but replace $\Z/p$ with $\Z$ and the first part of Lemma \ref{zeroout}(2) by
the second part of Lemma \ref{zeroout}(2).
\end{proof}

\section[Extensor Lemma]
{Extensor Lemma}\label{ExtenLem}

For the remainder of this paper, $I^\infty$ will denote the
Hilbert cube, i.e., $I^\infty=\prod\{I_i\,|\,i\in\N\}$ where
$I_i=[0,1]$ for each $i$.  For each $k\in\N$, we use $I^k$ to denote
$\{x\in I^\infty\,|\,x_i=0\,\,\mathrm{for}\,\,i>k\}$.
Of course if $1\leq j\leq k$, then $I^j\subset I^k$, and $p_j^k:I^k \to I^j$ denotes the coordinate projection.
If $P\subset I^k$, then
$P^{[\infty]}$ will designate the set of $x$ in $I^\infty$ whose first $k$ coordinates
are the same as those of an element of $P$.
Plainly if $P\subset Q\subset I^k$, then $P^{[\infty]}\subset
Q^{[\infty]}$. One sees that if $P$ is closed, respectively open, in $I^k$, then
$P^{[\infty]}$ is closed, respectively open, in $I^\infty$.
Let $p_{k,\infty}:I^\infty\to I^k$ denote the coordinate projection of the Hilbert cube to its finite part,
and we have that  $(p_{j,\infty}|I^k)\circ p_{k,\infty}=p_{j,\infty}:I^\infty\to I^j$.
It is therefore true that
$$(\dag_2)\,\,\,p_j^k\circ p_{k,\infty}=p_{j,\infty}:I^\infty\to I^j.$$
If $1\leq i\leq j\leq k$, then
$$(\dag_3)\,\,\,p_i^j\circ p_j^k=p_i^k:I^k\to I^i.$$

We shall use the metric $\rho$ on $I^\infty$ given by
$\rho(x,y)=\sum_{i=1}^{\infty}\frac{\vert x_i-y_i\vert}{2^i}$.
With this metric on $I^\infty$, $\diam I^\infty=1$. Moreover, if $1\leq j\leq k$
and $y\in I^k$, then
$$(\dag_4)\,\,\,\rho(p_j^k(y),y)<\frac{1}{2^j}.$$
Similarly, for any $x\in I^\infty$,
$$(\dag_5)\,\,\,\rho(p_{j,\infty}(x),x)\leq\frac{1}{2^j}.$$

The main result of this section is Lemma \ref{endset}, a type of ``extensor lemma.''  It provides us
with a statement, see (2) of that lemma, about extending a map under the condition that a given compactum $X$ is
a subspace of $I^\infty$.  A couple of preliminaries will be useful.

\begin{lemma}\label{closearehomtop}Let $K$ be a $\CW$-complex.
Then $K$ has an open cover $\mathcal{V}$ such that any two
$\mathcal{V}$-close maps of any space to $K$ are homotopic.
\end{lemma}

\begin{proof}There exists a simplicial complex $L$ such that
$|L|_m$, that is $|L|$ with the metric topology, is homotopy
equivalent to $K$.  Choose a homotopy equivalence
$f:K\to|L|_m$. By Theorem III.11.3 (page 106) of \cite{Hu},
$|L|_m$ is an ANR.  Theorem IV.1.1 (page 111) of \cite{Hu}
shows that there is an open cover $\mathcal{W}$ of $|L|_m$
having the property that any two $\mathcal{W}$-close maps of
any space to $|L|_m$ are homotopic.  The open cover needed for
$K$ is $\mathcal{V}=f\inv(\mathcal{W})$.
\end{proof}

\begin{lemma}\label{intersectinH} Let $X\subset I^\infty$ be compact and nonempty.
Then there exist an increasing sequence $(n_j)$ in $\N$ and a sequence $(P_j)$
of nonempty compact polyhedra $P_j\subset I^{n_j}$ such that for all $j\in\N:$
\begin{enumerate}
\item  $X\subset\inter_{I^\infty}(P_j^{[\infty]})\subset P_j^{[\infty]}
\subset N(X,\frac{2}{j})$,
\item $p_{n_j}^{n_{j+1}}(P_{j+1})\subset\inter_{I^{n_j}}P_j$, and
\item $P_{j+1}^{[\infty]}\subset\inter_{I^\infty}(P_j^{[\infty]})$.
\end{enumerate}

Then,

$(*_1)$  $X=\bigcap\{P_j^{[\infty]}\,|\,j\in\N\}$.

Suppose moreover, that $j\in\N$ and $B_j\subset P_j$.  Then,

$(*_2)$ if $k\geq j$, and $B_k=(p_{n_j}^{n_k})\inv(B_j)\cap P_k$, one has that
$B_k\subset P_k\subset I^{n_k}$ and $p_{n_j}^{n_k}(B_k)\subset B_j$, and

$(*_3)$ if we put $B_{j,\infty}=p_{n_j,\infty}\inv(B_j)\cap X$, then for any open
neighborhood $S$ of $B_{j,\infty}$ in $I^\infty$, there exists $k\geq j$ such
that for all $l\geq k$, $B_l\subset S$.

If the increasing sequence $(n_j)$ is replaced by an increasing subsequence, then
all of the above still hold true.
\end{lemma}

\begin{proof}We shall construct the sets $P_j$ by recursion. Put $n_1=1$ and $P_1=I^1$.
One sees that $P_1\subset P_1^{[\infty]}=I^\infty$, so $(1)$ is true in case $j=1$
since $\diam I^\infty=1$.  Because we are using a recursive
construction, when we come to a particular $j$, we only have to consider statements (1)--(3)
``up to $j$.'' Hence (2) and (3) need not be considered yet, and we will deal with
$(*_1)$--$(*_3)$ later.

Suppose that $j\in\N$, and we have found finite sequences $1=n_1<\dots<n_j$
in $\N$ and compact polyhedra $P_1,\dots, P_j$ such that
for $1\leq s\leq j$, $P_s\subset I^{n_s}$, and $(1)$--
$(3)$ are true up to $j$.

One may choose $n_{j+1}\in\N$ such that $n_{j+1}>n_j$ and $X\subset
(p_{n_{j+1},\infty}(X))^{[\infty]}\subset N(X,\frac{2}{j+1})$.
There is a neighborhood $V$ of $p_{n_{j+1},\infty}(X)$ in $I^{n_{j+1}}$ such that
$V^{[\infty]}\subset N(X,\frac{2}{j+1})$. Choose a compact
polyhedron $P_{j+1}$ of $I^{n_{j+1}}$ so that,
$p_{n_{j+1},\infty}(X)\subset\inter_{I^{n_{j+1}}} P_{j+1}
\subset P_{j+1}\subset V$.  From this, $X\subset(p_{n_{j+1},\infty}(X))^{[\infty]}
\subset(\inter_{I^\infty}P_{j+1})^{[\infty]}\subset P_{j+1}^{[\infty]}\subset V^{[\infty]}
\subset N(X,\frac{2}{j+1})$. This gives us $(1)$ for $j+1$.

Notice that $(1)$ for $j$ implies,
$p_{n_j,\infty}(X)=p_{n_j}^{n_{j+1}}\circ p_{n_{j+1},\infty}(X)\subset
\inter_{I^{n_j}}P_j$. Hence,
$p_{n_{j+1},\infty}(X)\subset(p_{n_j}^{n_{j+1}})^{-1}(\inter_{I^{n_j}}(P_j))$.
Thus, making $P_{j+1}$ smaller if necessary, we may have $(1)$ and
simultaneously, $P_{j+1}\subset(p_{n_j}^{n_{j+1}})^{-1}(\inter_{I^{n_j}}(P_j))$.
This achieves $(2)$ for $j+1$. From (2) we get (3).

One gets $(*_1)$ readily from (1), and $(*_2)$ is true for elementary reasons.
To prove $(*_3)$ let $S$ be an open neighborhood of $B_{j,\infty}$ in
$I^\infty$.  If the conclusion of $(*_3)$ is not true, then there
is an increasing sequence $(m_i)$ in $\N$, $m_1\geq j$, so that
for each $i$, there exists $b_i\in B_{m_i}\setminus S
\subset P_{m_i}$.  Passing to a subsequence if necessary,
we may assume that the sequence $(b_i)$ in the compactum
$I^\infty\setminus S$ converges in $I^\infty$ to $b\in I^\infty\setminus S$.
Applying $(*_1)$ and $(*_2)$ along with the fact that $b_i\in P_{m_i}$, one sees that $b\in X
\setminus B_{j,\infty}$, from which we deduce that $p_{n_j,\infty}(b)\notin B_j$.

For each $i$, $p_{n_j}^{s_i}(b_i)=p_{n_j,\infty}(b_i)\in B_j$, $s_i=n_{m_i}$.
Hence $\{p_{n_j,\infty}(b_i)\,|\,i\in\N\}\subset B_j$.  Since $B_j$ is closed in
$P_j$, $p_{n_j,\infty}$ is a map, and $(b_i)$ converges to $b$, then $p_{n_j,\infty}(b)\in B_j$, a
contradiction.  This yields $(*_3)$.  We leave the validation of the final statement
to the reader.
\end{proof}

In reading Lemma \ref{endset}, one should consult Lemma \ref{intersectinH}(2)
to see that whenever $j\leq l$, then $p_{n_j}^{n_l}(P_l)\subset P_j$.

\begin{lemma}\label{endset}Let $X\subset I^\infty$, $(n_j)$, $(P_j)$ be as in
$\mathrm{Lemma\,\,\ref{intersectinH}}$, and $K$ be a $\CW$-complex
with $X\tau K$.   The following are true.
\begin{enumerate}
\item Suppose that $j\in\N$ and $B_j$ is a nonempty closed subset of $P_j$.
For each $k\geq j$, let $B_k=(p_{n_j}^{n_k})\inv(B_j)\cap P_k$.
If $M$ is a finite subcomplex of $K$ and $f:B_j\to M$ is
a map, then there exists $k\geq j$ such that for all $l\geq k$, there
is a map $f^*:P_l\to K$ that extends the composition $f\circ
p_{n_j}^{n_l}|B_l:B_l\to M$, where we treat $p_{n_j}^{n_l}
|B_l:B_l\to B_j$.
\item  Suppose that $K\in\{K(\Z,n),K(\Z/p,n),K(\Q,n)\}$ and $L_j$ is a triangulation of
$P_j$.  Then there exists $k\geq j$
such that for all $l\geq k$, there is a triangulation $N_l$ of $P_l$,
a simplicial approximation $\overline p_j:|N_l|\to|L_j|$ of $p_{n_j}^{n_l}:P_l\to P_j$, and
a map $g_K:|N_l^{(n+1)}|\to|L_j^{(n+1)}|$ as in
$\mathrm{Proposition\,\,\ref{pushandshove2}}$ where $(E,N,L,\overline p,\overline p_1)$
of that proposition is replaced by $(P_l,N_l,L_j,p_{n_j}^{n_l},\overline p_j)$ in
the present context.
\end{enumerate}
\end{lemma}

\begin{proof}We shall first prove (1).
Observe that $B_{j,\infty}=p_{n_j,\infty}\inv(B_j)\cap X$ is a closed subset of $X$ and that
$p_{n_j,\infty}(B_{j,\infty})\subset B_j$. Since $X\tau K$, then
the map $f\circ p_{n_j,\infty}|B_{j,\infty}:B_{j,\infty}\to M$
extends to a map $h:X\to\widetilde M$ where $\widetilde M$ is a
finite subcomplex of $K$ and $M\subset\widetilde M$. Since
$\widetilde M$ is an ANR, we may additionally
assume that there is an open neighborhood $U$
of $X$ in $I^\infty$ and that $h:U\to\widetilde M$.

Employing Lemma \ref{closearehomtop},
select an open cover $\mathcal{V}$ of $\widetilde M$
such that for any space $Y$, any maps $g_1:Y\to\widetilde M$
and $g_2:Y\to\widetilde M$
that are $\mathcal{V}$-close are homotopic.  Let $\mathcal{V}_1$
be an open cover of $\widetilde M$ that star-refines $\mathcal{V}$.
Choose an open cover $\mathcal{W}$ of $B_j$ such that if
$W\in\mathcal{W}$, then there
exists $V_W \in\mathcal{V}_1$ with $f(W)\subset V_W$.

Select a cover $\mathcal{R}$ of
$X$ by sets open in $U$ having the property that for each
$R\in\mathcal{R}$, there exists $V_R\in\mathcal{V}_1$ with
$h(R)\subset V_R$.  Let $S=\bigcup\mathcal{R}\subset U$.  Then
$S$ is an open neighborhood of $X$ in $I^\infty$, and of course $B_{j,\infty}\subset X$.  So
by $(*_3)$ of Lemma \ref{intersectinH}, we may choose $k_0\in\N$ so that $k_0\geq j$
and for all $l\geq k_0$,
$B_l\subset S$.  Using Lemma \ref{intersectinH}(3) and Lemma \ref{intersectinH}($(*_1)$), we may
also require that for such $l$, $P_l\subset S$.

Put $B^*=B_{j,\infty}\cup\bigcup\{B_l\,|\,l\geq k_0\}\subset S$.  An
application of Lemma \ref{intersectinH}($(*_2)$) yields that $p_{n_j,\infty}|B^*:B^*\to B_j$ is a map.
For each $b\in B_{j,\infty}$,
select an open neighborhood $E_b$ of $b$ in $B^*$ such that $p_{n_j,\infty}(E_b)$
is contained in an element $W_b$ of $\mathcal{W}$ and that in
addition, there exists $R_b\in\mathcal{R}$ with $E_b\subset R_b$.
Let $S_0=\bigcup\{E_b\,|\,b\in B_{j,\infty}\}$.  Then $S_0$ is an
open neighborhood of $B_{j,\infty}$ in $B^*\subset I^\infty$.  So
there is an open subset $S_1$ of $I^\infty$ having the property that $S_1\cap B^*=S_0$.
Plainly, $S_1$ is an open neighborhood of $B_{j,\infty}$ in $I^\infty$.
An application of $(*_3)$ of Lemma \ref{intersectinH}, with $S_1$ in place of $S$ gives
us the existence of a $k\geq k_0$ so that for all $l\geq k$, we have $B_l\subset S_1$.
For such $l$, $B_l\subset S_1\cap B^*=S_0$.

We are going to show that for all $l\geq k$,
$f\circ p_{n_j}^{n_l}|B_l:B_l\to|M|\subset|\widetilde M|$
is homotopic to $h|B_l:B_l\to|\widetilde M|$.
For in that case, if we define $h_0=h|B_l:B_l\to|\widetilde M|$, then of course
since $P_l\subset S$, $h_0$ extends to the map
$h|P_l:P_l\to|\widetilde M|$, and the homotopy
extension theorem will complete our proof of (1).

Let $x\in B_l$.  It will be sufficient to show that
$f\circ p_{n_j}^{n_l}(x)$ and $h_0(x)$
lie in an element of $\mathcal{V}$.
There exists $b\in B_{j,\infty}$ such that $x\in E_b$.  Now
$b\in E_b\subset R_b\in\mathcal{R}$.  It follows that
$h(\{b,x\})=\{h(b),h(x)\}=\{h(b),h_0(x)\}\subset V_1=V_{R_b}
\in\mathcal{V}_1$.  One sees from the
definition of $h$ and the fact
that $b\in B_{j,\infty}$, that $h(b)=f\circ p_{n_j,\infty}(b)$.
So we have that $\{f\circ p_{n_j,\infty}(b),h_0(x)\}
\subset V_1\in\mathcal{V}_1$.
We know that $p_{n_j,\infty}(b)\in p_{n_j,\infty}(E_b)\subset W_b\in\mathcal{W}$.
Thus $f\circ p_{n_j,\infty}(b)\in f\circ p_{n_j,\infty}(E_b)
\subset f(W_b)\subset V_2=V_{W_b}\in\mathcal{V}_1$.
Now $x\in E_b\cap B_l\subset E_b\cap P_l\subset
E_b\cap I^{n_l}$, so $p_{n_j,\infty}(x)=p_{n_j}^{n_l}(x)\in
p_{n_j,\infty}(E_b)$, and we see that $f\circ p_{n_j}^{n_l}(x)\in
f\circ p_{n_j,\infty}(E_b)\subset V_2$.
Hence, $\{f\circ p_{n_j,\infty}(b),f\circ p_{n_j}^{n_l}(x)\}\subset
V_2$.  Since $\mathcal{V}_1$ is a star-refinement of $\mathcal{V}$,
$f\circ p_{n_j,\infty}(b)\in V_1\cap V_2$, $h_0(x)\in V_1$,
and $f\circ p_{n_j}^{n_l}(x)\in V_2$,
one may find $V\in\mathcal{V}$ with $\{f\circ p_{n_j}^{n_l}(x),h_0(x)\}
\subset V_1\cup V_2\subset V$.  Our proof of (1) is complete.

Now we must prove (2). Let $S$ be an $n$-regular subdivision of $L_j$.  Suppose
that $\sigma\in L_j$ with $\dim\sigma=n+1$.  Let $M=|N(\partial\sigma,S)|$,
and treat $M$ as a finite $\CW$-subcomplex of $K_{\sigma,+}$.
By Lemma \ref{wearaskirt}(3), $X\tau K_{\sigma,+}$.
Denote $f_\sigma=\id:M\to M$, and put $B_j=M$.  Apply (1) to these data to get $k_\sigma\geq j$
as indicated there.  Put $k=\max\{k_\sigma\,|\,(\sigma\in L_j)\wedge(\dim\sigma=n+1)\}$.
Hence for all $\sigma\in L_j$ with $\dim\sigma=n+1$ and
$l\geq k$, there is a map $f_\sigma^*:P_l\to K_{\sigma,+}$
that extends the composition $f_\sigma\circ p_{n_j}^{n_l}|(p_{n_j}^{n_l})\inv(|N(\partial\sigma,S)|):
(p_{n_j}^{n_l})\inv(|N(\partial\sigma,S)|)\to|N(\partial\sigma,S)|$.

Suppose that $l\geq k$.  Let $E=P_l$, $\overline p=p_{n_j}^{n_l}|P_l:P_l\to P_j$,
and $N_l$ be a triangulation of $P_l$ that
admits a simplicial approximation $\overline p_j:|N_l|\to|L_j|$ of the map $\overline p$.
Now apply Proposition \ref{pushandshove2} to get a map $g_K:|N_l^{(n+1)}|\to|L_j^{(n+1)}|$
as requested.
\end{proof}

\section[Epsilons, Deltas, Maps, and Fibers]
{Epsilons, Deltas, Maps, and Fibers}\label{Technicallemma}

We rely on the notation established in the
first part of Section \ref{ExtenLem} concerning the Hilbert cube $I^\infty$
and its metric $\rho$.  Throughout this section $X\subset I^\infty$
will denote a nonempty compactum.
Assume in addition that we are given an increasing sequence $(n_j)$ in $\N$ and a sequence
$(P_j)$ of nonempty compact polyhedra $P_j\subset I^{n_j}$ satisfying $(1)$-$(3)$ of $\mathrm{Lemma\,\,
\ref{intersectinH}}$. The following technical lemma is similar to Lemma 3.1
of \cite{AJR} (repeated in \cite{RT} as Lemma 2.1).

\begin{lemma}\label{technical}
Suppose that for each $j\in\N$ we have selected a closed subset $T_j$ of $P_j$, $\delta_j>0$,
$\epsilon_j>0$, and a map $g_j^{j+1}:T_{j+1}\to T_j$ so that:
\begin{enumerate}
\item  if $u$, $v\in I^\infty$ and
$\rho(u,v)\leq\epsilon_{j+1}$, then $\rho(p_{n_j,\infty}(u),p_{n_j,\infty}(v))<\delta_j$,
\item $9\cdot2^{-n_j}<\epsilon_j$,
\item $\rho(g_j^{j+1}(u),p_{n_j}^{n_{j+1}}(u))<\delta_j$ for all $u\in T_{j+1}$, and
\item $0<\delta_j<2^{1-n_j}=2\cdot2^{-n_j}$.
\end{enumerate}

Put $\mathbf T=(T_j,g_j^{j+1})$, and $Z=\lim\mathbf T$.
Then $Z$ is a metrizable compactum, and
for each $z=(a_1,a_2,\dots)\in Z\subset\prod_{j=1}^{\infty}T_j$, \begin{enumerate}
\item[$(*_1)$] the associated sequence $(a_j)$
is a Cauchy sequence in $I^\infty$ whose limit lies in $X$, and
\item[$(*_2)$] the function $\pi:Z\to X$ given by $ \displaystyle
\pi(z)= \lim_{j\to\infty}(a_j)$ is a map.
\end{enumerate}

Fix $x\in X$ and for each $j\in\N$, let $B_{x,j}=\overline
N(p_{n_j,\infty}(x),2\delta_j)\cap T_j, B_{x,j}^\#=\overline
N(p_{n_j,\infty}(x),\epsilon_j)\cap T_j$. Then,
\begin{enumerate}
\item[$(*_3)$] $B_{x,j}$, $B_{x,j}^\#$ are closed in $T_j$
and $B_{x,j}\subset B_{x,j}^\#\subset T_j\subset P_j$, and
\item[$(*_4)$]$g_j^{j+1}(B_{x,j+1}^\#)\subset B_{x,j}$.
\end{enumerate}

If we let $\bold T_x=(B_{x,j},g_j^{j+1})$ and $\bold T_x^\#=
(B_{x,j}^\#,g_j^{j+1})$, then, \begin{enumerate}
\item[$(*_5)$] $\lim\bold T_x=\lim\bold T_x^\#\subset Z$,
\item[$(*_6)$] $\pi^{-1}(x)=\lim\bold T_x$,
\item[$(*_7)$] if for each $j$ one has chosen $E_{x,j}$ with $B_{x,j}\subset
E_{x,j}\subset B_{x,j}^\#$, then $g_j^{j+1}(E_{x,j+1})\subset E_{x,j}$ and
with $\bold E_x=(E_{x,j},g_j^{j+1})$, $\pi^{-1}(x)=\lim\bold E_x=\lim\bold T_x$, and lastly
\item[$(*_8)$] if $B_{x,j}\neq\emptyset$ for each $j$, then
$\pi^{-1}(x)\neq\emptyset$.
\end{enumerate}
\end{lemma}

\begin{proof}Since each $T_j$ is a metrizable compactum, then by Theorem \ref{presdim}(1),
$Z=\lim\mathbf T$ is a metrizable compactum.
Employing $(\dag_5)$ of Section \ref{ExtenLem}, we have,

$(\dag_1)$  for all $x\in I^\infty$ and $j\in\N$, $\rho(p_{n_j,\infty}(x),x)
=\rho(p_{n_j}^{n_{j+1}}\circ p_{n_{j+1},\infty}(x),x)\leq 2^{-n_j}$.

The triangle inequality along with $(\dag_1)$, (3), and (4) show
that, independently of choice of $z=(a_1,a_2,\dots)\in Z$, for all $j\in\N$
$\rho(a_j,a_{j+1})=\rho(g_j^{j+1}(a_{j+1}),a_{j+1})
\leq\rho(g_j^{j+1}(a_{j+1}),p_{n_j}^{n_{j+1}}(a_{j+1}))
+\rho(p_{n_j}^{n_{j+1}}(a_{j+1}),a_{j+1})=\rho(g_j^{j+1}(a_{j+1}),p_{n_j}^{n_{j+1}}(a_{j+1}))
+\rho(p_{n_j,\infty}(a_{j+1}),a_{j+1})<\delta_j+2^{-n_j}<2^{1-n_j}+2^{-n_j}=
2\cdot2^{-n_j}+2^{-n_j}=3\cdot2^{-n_j}<2^2\cdot2^{-n_j}$, so

$(\dag_2)$ independently of choice of $z=(a_1,a_2,\dots)\in Z$, for all
$j\in\N$ $\rho(a_j,a_{j+1})<2^{2-n_j}$.

From $(\dag_2)$ and the fact that $(n_j)$ is increasing, such $(a_j)$ is a Cauchy sequence in $I^\infty$.
An application of Lemma \ref{intersectinH}(3) shows that
for each $j$, $P_{j+1}^{[\infty]}\subset P_j^{[\infty]}$.
By Lemma \ref{intersectinH} $(*_1)$, $X=\bigcap_{j=1}^{\infty}P_j^{[\infty]}
\subset I^\infty$.  Since $a_j\in T_j\subset P_j\subset P_j^{[\infty]}$
for each $j$, one concludes the validity of $(*_1)$.
The statement $(*_2)$ is true since $(\dag_2)$ shows that $\pi$ is the
limit of the uniformly convergent sequence of maps $\pi_j:Z\to I^{n_j}\subset I^\infty$ where
$\pi_j(z)=a_j$ whenever $z=(a_1,a_2,\dots)\in Z$.
We prove $(*_3)$ by noting that (4) and (2) imply that $2\delta_j<\epsilon_j$
so that $B_{x,j}\subset B_{x,j}^\#$.

Next let $u\in
B_{x,j+1}^\#\subset\overline N(p_{n_{j+1},\infty}(x),\epsilon_{j+1})$.  Thus,

$(\dag_3)$  $\rho(u,p_{n_{j+1},\infty}(x))\leq\epsilon_{j+1}$, and $u\in T_{j+1}$.

By hypothesis, $g_j^{j+1}(u)\in T_j$.  So it remains to prove that

$(\dag_4)$  $\rho(g_j^{j+1}(u),p_{n_j,\infty}(x))<2\delta_j$.

As a consequence of (1) and the first part of $(\dag_3)$,
$\rho(p_{n_j,\infty}(u),p_{n_j,\infty}\circ p_{n_{j+1},\infty}(x))=
\rho(p_{n_j}^{n_{j+1}}(u),p_{n_j,\infty}(x))<\delta_j$. This,
the triangle inequality, and (3) show that $\rho(g_i^{j+1}(u),p_{n_j,\infty}(x))\leq
\rho(g_i^{j+1}(u),p_{n_j,\infty}(u))+\rho(p_{n_j,\infty}(u),\break p_{n_j,\infty}(x))=
\rho(g_j^{j+1}(u),p_{n_j}^{n_{j+1}}(u))+\rho(p_{n_j}^{n_{j+1}}(u),p_{n_j,\infty}(x))
<\delta_j+\delta_j=2\delta_j$.  This validates $(\dag_4)$ and hence
establishes $(*_4)$. Item $(*_5)$ is an immediate consequence of $(*_3)$ and $(*_4)$.

We now want to prove $(*_6)$, that $\pi^{-1}(x)=\lim\bold T_x$.  If
$a=(a_1,a_2,\dots)$ is a thread of $\bold T_x$, then for $j\in\N$, $a_j\in B_{x,j}$,
and hence $\rho(a_j,p_{n_j,\infty}(x))\leq2\delta_j$. By applying this, $(\dag_1)$,
and (4), $\rho(a_j,x)\leq\rho(a_j,p_{n_j,\infty}(x))+\rho(p_{n_j,\infty}(x),x)
\leq2\delta_j+2^{-n_j}\leq5\cdot2^{-n_j}$.  Hence, $\pi(a)=\lim(a_j)=x$, so
$\lim\bold T_x\subset\pi\inv(x)$.

Towards the opposite inclusion, suppose that a thread $(a_1,a_2,\dots)$
of $\bold T$ lies in $\pi^{-1}(x)$.  Apply the triangle inequality, the
fact that $(a_j)$ converges to $x$, $(\dag_1)$, $(\dag_2)$, (4), and (2) to see that when $j>1$,
$\rho(a_j,p_{n_j,\infty}(x))\leq\rho(a_j,x)+\rho(x,p_{n_j,\infty}(x))\leq
\sum_{k=j}^\infty\rho(a_k,a_{k+1})+2^{-n_j}<\sum_{k=j}^\infty2^{2-n_k}+2^{-n_j}\leq2\cdot2^{2-n_j}
+2^{-n_j}=9\cdot2^{-n_j}<\epsilon_j$.  This puts $a_j\in B_{x,j}^\#$.  So
$(a_1,a_2,\dots)\in\lim\bold T_x^\#=\lim\bold T_x$, showing that $\pi^{-
1}(x)\subset\lim\bold T_x$.  Hence $\pi^{-1}(x)=\lim\bold T_x$ as we had
proclaimed.

The statement $(*_7)$ is plainly true.  Employing Theorem \ref{presdim}(3) and $(*_6)$,
one sees that if $B_{x,j}\neq\emptyset$ for each $j$, then $\lim\bold T_x=\pi\inv(x)
\neq\emptyset$, which establishes $(*_8)$.
\end{proof}

\begin{lemma}\label{squeezeplay}  Let $j\in\N$, $\epsilon_j>0$, $N_j$ a triangulation
of $P_j$ with $2\cdot\mesh N_j<\epsilon_j$, $\lambda_j$ be a Lebesgue number of the open cover
$\mathcal{N}_j=\{\st(v,N_j)\,|\,v\in N_j^{(0)}\}$ of $P_j$, $\delta_j>0$, and
$4\cdot\delta_j<\lambda_j$.  Then for all $x\in X$, there exists $v_{x,j}\in\N_j^{(0)}$ such
that $\overline N(p_{n_j,\infty}(x),2\delta_j)\subset\overline\st(v_{x,j},N_j)\subset
\overline N(p_{n_j,\infty}(x),\epsilon_j)$.
\end{lemma}

\begin{proof}By $(*_1)$ of Lemma \ref{intersectinH}, $x\in
P_j^{[\infty]}$, so $p_{n_j,\infty}(x)\in P_j$. It follows from the
fact that $4\cdot\delta_j<\lambda_j$ and the
definition of $\mathcal{N}_j$ that there exists $v_{x,j}\in N_j^{(0)}$ such that
$\overline N(p_{n_j,\infty}(x),2\delta_j)\subset\overline\st(v_{x,j},N_j)$.
Now, $p_{n_j,\infty}(x)\in\overline\st(v_{x,j},N_j)$, and $2\cdot\mesh(N_j)<\epsilon_j$.
This implies that $\overline\st(v_{x,j},N_j)\subset\overline N(p_{n_j,\infty}(x),\epsilon_j)$.
\end{proof}

\begin{lemma}\label{newtechnical}Let $n\in\N$, $K\in\{K(\Z,n),K(\Z/p,n),K(\Q,n)\})$,
and $j\in\N$. Suppose that $\epsilon_j>0$ and a triangulation $N_j$
of $P_j$ have been chosen so that $2\cdot\mesh(N_j)<\epsilon_j$.
Then there exist $\delta_j$, $L_j$, and $\varphi_j$, and for any given $m\in\N$,
there exist $l\in\N_{\geq m}$ and $\epsilon_l$ such that:

\begin{enumerate}
\item $0<\delta_j<2^{1-n_j}=2\cdot2^{-n_j}$,

\item $L_j$ is a subdivision of $N_j$ with $\mesh L_j<\delta_j$,

\item$\varphi_j:|L_j|\to|N_j|$ is a simplicial approximation of $\id_{P_j}$,

\item for each $x\in X$, there exists $v_{x,j}\in N_j^{(0)}\subset L_j^{(0)}$ such that
we have $\overline N(p_{n_j,\infty}(x),2\delta_j)\subset\overline\st(v_{x,j},N_j)\subset
\overline N(p_{n_j,\infty}(x),\epsilon_j),$

\item $j<l$,

\item$9\cdot2^{-n_l}<\epsilon_l$,
\item if $u$, $v\in I^\infty$ and
$\rho(u,v)\leq\epsilon_l$, then $\rho(p_{n_j,\infty}(u),p_{n_j,\infty}(v))<\delta_j$,

\item $N_l$ is a triangulation of $P_l$ with
$2\cdot\mesh N_l<\epsilon_l$ that admits a simplicial approximation
$\overline p_j:|N_l|\to|L_j|$ of $p_{n_j}^{n_l}:P_l\to P_j$, and

\item there is a map $g_K=g_{j,K}:|N_l^{(n+1)}|\to|L_j^{(n+1)}|$
as in $\mathrm{Lemma\,\,\ref{endset}(2)}$ with respect to the current
$(P_l,N_l,L_j,p_{n_j}^{n_l},\overline p_j)$,

\item for any subdivision $L_l$ of $N_l$ and simplicial approximation
$\varphi_l:|L_l|\to|N_l|$ of $\id_{P_l}$,
$g_{j,K}(\varphi_l(|L_l^{(s)}|))\subset|L_j^{(s)}|$ for $s\in\{n,n+1\}$,

\item for all $x\in|L_l^{(n+1)}|$, there exists $\sigma\in L_j^{(n+1)}$ such that
$\{p_{n_j}^{n_l}(x),g_K\circ\varphi_l(x)\}\subset\sigma$, and

\item for all $x\in|L_l^{(n+1)}|$,
$\rho(p_{n_j}^{n_l}(x),g_K\circ\varphi_l(x))<\delta_j$.
\end{enumerate}
\end{lemma}

\begin{proof} Let $\lambda_j$ be a Lebesgue number of the open cover
$\mathcal{N}_j=\{\st(v,N_j)\,|\,v\in N_j^{(0)}\}$ of $P_j$, and find
$\delta_j>0$ so that (1) and $4\cdot\delta_j<\lambda_j$ are true.
Select $L_j$ as in (2) and using Lemma \ref{sipapprox}(1), obtain $\varphi_j$
as needed in (3).  According to Lemma \ref{squeezeplay},
(4) holds true.  Using Lemma \ref{endset}(2), one can obtain $l$ so that (5)-(9) are in
effect.  Proposition \ref{pushandshove2}(2) gives us (10), (10) and
Proposition \ref{pushandshove2}(3) give us (11); (11) and (2) lead to (12).
\end{proof}

\section[Recursion]
{Recursion}\label{Recursion}

\begin{lemma}\label{newtechnicalrecur}Let $X\subset I^\infty$ be a nonempty compactum
and $n\in\N$. Suppose that $(n_j)$ is an increasing sequence in $\N$, $(P_j)$ is a sequence
of nonempty compact polyhedra $P_j\subset I^{n_j}$ satisfying $(1)$-$(3)$ of $\mathrm{Lemma\,\,
\ref{intersectinH}}$, and $$K\in\{K(\Z,n),K(\Z/p,n),K(\Q,n)\})$$ with $X\tau K$.
Let $\epsilon_1=2^{-1}\cdot9\cdot2^{-n_1}$, and $N_1$ be a triangulation of
$P_1$ with $2\cdot\mesh N_1<\epsilon_1$.  Then
there exist a function $l:\N\to\N$ with $l(1)=1$, a sequence
$(\epsilon_j,N_j,\delta_j,L_j,\varphi_j,\overline p_j,g_{j,K})$,
and for each $x\in X$, a sequence $(v_{x,j},N_{x,j},L_{x,j})$ such that for all $j\in\N$,
\begin{enumerate}

\item $0<\delta_j<2^{1-n_{l(j)}}=2\cdot2^{-n_{l(j)}}$,

\item $N_j$ is a triangulation of $P_{l(j)}$ and $L_j$ is a subdivision
of $N_j$ with $\mesh L_j<\delta_j$,

\item$\varphi_j:|L_j|\to|N_j|$ is a simplicial approximation to $\id_{P_{l(j)}}$,

\item$v_{x,j}\in N_j^{(0)}\subset L_j^{(0)}$ and we have $\overline N(p_{n_{l(j)},\infty}(x),2\delta_j)\subset\overline
\st(v_{x,j},N_j)\subset\overline N(p_{n_{l(j)},\infty}(x),\epsilon_j)$,

\item $l(j)<l(j+1)$,

\item$9\cdot2^{-n_{l(j)}}<\epsilon_j$,

\item if $u$, $v\in I^\infty$ and
$\rho(u,v)\leq\epsilon_{j+1}$, then $\rho(p_{n_{l(j)},\infty}(u),p_{n_{l(j)},\infty}(v))<\delta_j$,

\item$N_j$ is a triangulation of $P_{l(j)}$ with $2\cdot\mesh N_j<\epsilon_j$
that admits a simplicial approximation
$\overline p_j:|N_{j+1}|\to|L_j|$ of $p_{n_{l(j)}}^{n_{l(j+1)}}:P_{l(j+1)}\to P_{l(j)}$,

\item there is a map $g_{j,K}:|N_{j+1}^{(n+1)}|\to|L_j^{(n+1)}|$ as in
$\mathrm{Lemma\,\,\ref{endset}(2)}$ with $(P_{l(j+1)},N_{j+1},L_j,
p_{n_{l(j)}}^{n_{l(j+1)}},\overline p_j)$ playing the role of $(P_l,N_l,L_j,p_{n_j}^{n_l},
\overline p_j)$ of that lemma,

\item $g_{j,K}(\varphi_{j+1}(|L_{j+1}^{(s)}|))\subset|L_j^{(s)}|$ for $s\in\{n,n+1\}$,

\item if $x\in|L_{j+1}^{(n+1)}|$, there exists $\sigma\in L_j^{(n+1)}$ such that
$\{p_{n_{l(j)}}^{n_{l(j+1)}}(x),g_{j,K}\circ\varphi_{j+1}(x)\}\subset\sigma$, and

\item if $x\in|L_{j+1}^{(n+1)}|$, $\rho(p_{n_{l(j)}}^{n_{l(j+1)}}(x),g_{j,K}
\circ\varphi_{j+1}(x))<\delta_j$.

\end{enumerate}
\end{lemma}

\begin{proof}Define $l(1)=1$ and $\epsilon_1=2^{-1}\cdot9\cdot2^{-n_{l(1)}}$; select a triangulation $N_1$ of
$P_{l(1)}$ with $2\cdot\mesh N_1<\epsilon_1$.  Then apply Lemma
\ref{newtechnical} recursively to obtain this result.
\end{proof}

\begin{lemma}\label{starfacts} Under the hypothesis of $\mathrm{Lemma\,\,\ref{newtechnicalrecur}}$,
for each $j\in\N$, let $N_{x,j}$, $L_{x,j}$ respectively
be defined as the subcomplexes of $N_j$ and $L_j$ such that $\overline\st(v_{x,j},N_j)=
|N_{x,j}|=|L_{x,j}|$.  Then,
\begin{enumerate}
\item$\overline\st(v_{x,j},N_j)=|N_{x,j}|=|L_{x,j}|$ is a contractible
metrizable continuum,

\item for all $s\geq 0$, $\overline\st(v_{x,j},N_j)\cap|L_j^{(s)}|=|L_{x,j}^{(s)}|$,

\item$v_{x,j}\in\overline\st(v_{x,j},N_j)\cap|L_{x,j}^{(0)}|$,
so $\overline\st(v_{x,j},N_j)\cap|L_{x,j}^{(0)}|\neq\emptyset$,

\item for all $s\geq 1$, $|L_{x,j}^{(s)}|$ is a metrizable continuum and
for all $s\geq 0$, $|L_{x,j}^{(s)}|$ contracts to a point in $|L_{x,j}^{(s+1)}|$,

\item for any abelian group $G$ and $1\leq s<n$, $H_s(|L_{x,j}^{(n)}|;G)=0$,

\item for any abelian group $G$ and $s>n$, $H_s(|L_{x,j}^{(n)}|;G)=0$,

\item $H_n((|L_{x,j}^{(n)}|;\Z)$ is free abelian of finite rank, and

\item $H_0(|L_{x,j}^{(n)}|;\Z)$ is isomorphic to $\Z$, so it is free abelian
of finite rank.
\end{enumerate}
\end{lemma}

\begin{proof}As mentioned in
Section \ref{stars}, $\overline\st(v_{x,j},N_j)$ is a contractible closed subset of the
compact polyhedron $|N_j|$, so it is a contractible metrizable continuum
and we get (1). It is plain that (2) is true.
Surely $v_{x,j}\in\overline\st(v_{x,j},N_j)\cap|N_{x,j}^{(0)}|\subset\overline\st(v_{x,j},N_j)
\cap|L_{x,j}^{(0)}|$, so we have (3).
It follows from (1) that $|L_{x,j}^{(s)}|$ is connected, and since it is
closed in $|L_{x,j}|$, it is a metrizable continuum.  The rest of (4) is a result of
Lemma \ref{ktok+1}.  We get (5) from (1) and Lemma \ref{contractsub}, and (6) is of
course true since $\dim|L_{x,j}^{(n)}|\leq n<s$.  Lemma \ref{isfreeab} gives us (7).
We get (8) from (4).
\end{proof}

\section[Proof of Main Results]
{Proof of Main Results}\label{proofofmain}

We now provide our coordinated proofs of Theorems \ref{cellresol}, \ref{Zmodpreol}, and \ref{rationalres}.

\begin{proof}Let $K\in\{K(\Z,n),K(\Z/p,n),K(\Q,n)\})$, $n\in\N$, and $X\subset I^\infty$ be
a compactum with $X\tau K$. The reader should keep in mind that by
Lemma \ref{equivforcohdim}, $\dim_G X\leq n$ is equivalent to $X\tau K(G,n)$.
Use the notation and facts that were developed in Lemmas \ref{newtechnicalrecur}
and \ref{starfacts} as applied to the current $X$ and $n$.
From Lemma \ref{newtechnicalrecur}(5), one sees that

$(\dag_1)$ $(n_{l(j)})$ is an increasing
sequence in $\N$.

For each $j\in\N$, let $T_j=|L_j^{(n)}|$. Then,
an application of  Lemma \ref{newtechnicalrecur}(2)
to the fact that $|L_j|=P_{l(j)}\neq\emptyset$ shows that

$(\dag_2)$ for all $j\in\N$, $T_j$ is a nonempty compact polyhedron
with $\dim T_j\leq n$.

Using Lemma \ref{newtechnicalrecur}(10), define $$g_j^{j+1}
=g_{j,K}\circ\varphi_{j+1}|T_{j+1}:T_{j+1}\to T_j.$$
This gives us an inverse sequence $\mathbf{T}=(T_j,g_j^{j+1})$ (depending on $K$)
of nonempty compact polyhedra (see $(\dag_2)$).
Let $Z=\lim\mathbf{T}$.  The first thing to observe is that since $T_j$ is a
nonempty compact polyhedron and $\dim T_j\leq n$ for each $j$, then by Theorem \ref{presdim},

$(\dag_3)$  $Z$ is a nonempty metrizable compactum, and $\dim Z\leq n$.

Items (7), (6), (12), and (1) of  Lemma \ref{newtechnicalrecur}, respectively, give
us (1)-(4) of Lemma \ref{technical}, which then provides us with a map $\pi:Z\to X$.
An internal description of the fibers of $\pi$ is now in order.

Now fix $x\in X$.  Using Lemma \ref{newtechnicalrecur}(4)
and Lemma \ref{starfacts}(2), one sees that for
all $j\in\N$, $\overline N(p_{n_{l(j)},\infty}(x),2\cdot\delta_j)\cap T_j=
\overline N(p_{n_{l(j)},\infty}(x),2\cdot\delta_j)\cap|L_j^{(n)}|\subset\overline\st(
v_{x,j},N_j)\cap T_j=\overline\st(v_{x,j},N_j)\cap|L_j^{(n)}|=|L_{x,j}^{(n)}|\subset
\overline N(p_{n_{l(j)},\infty}(x),\epsilon_j)\cap T_j$.  The leftmost and rightmost
sets in the preceding expression correspond respectively to $B_{x,j}$ and $B_{x,j}^{\#}$
in Lemma \ref{technical}.  Looking at $(*_7)$ of Lemma \ref{technical} with $E_{x,j}=|L_{x,j}^{(n)}|$,
one sees that $\mathbf{E}_x=(|L_{x,j}^{(n)}|,g_j^{j+1})$ is an inverse sequence such
that $\pi\inv(x)=\lim\mathbf{E}_x$.  An application of Lemma \ref{starfacts}(3,4) shows that
each coordinate space in $\mathbf{E}_x$ is nonempty and each is a metrizable continuum; therefore,

$(\dag_4)$ for each $x\in X$, $\pi\inv(x)$ is a nonempty metrizable continuum.

{\bf Proof of Theorem \ref{cellresol}.}  This time $K=K(\Z,n)$.
Using \ref{newtechnicalrecur}(9), we may apply
Proposition \ref{pushandshove2}(8) to see that each $g_j^{j+1}$ in $\mathbf{E}_x$ is
null homotopic. Apply Lemma \ref{trivshape} to complete this part of the proof.

{\bf Proof of Theorems \ref{Zmodpreol} and \ref{rationalres}.} In these cases we have that
$K\in\{K(\Z/p,n),K(\Q,n)\})$. By Proposition \ref{pushandshove2}(9,10)

$(\dag_5)$   if $K=K(\Z/p,n)$, then the induced homomorphism
$H_n(g_j^{j+1};\Z/p):H_n(|L_{x,j+1}^{(n)}|;\Z/p)\to
H_n(|L_{x,j}^{(n)}|;\Z/p)$ is trivial, and

$(\dag_6)$ if $K=K(\Q,n)$, then the induced homomorphism $H_n(g_j^{j+1};\Z):H_n(|L_{x,j+1}^{(n)}|;\Z)\to
H_n(|L_{x,j}^{(n)}|;\Z)$ is trivial.

Our proof of Theorem \ref{Zmodpreol} follows from $(\dag_5)$, Lemma \ref{starfacts}(5,6,7)
(Lemma \ref{starfacts}(8) in case $n=1$),
and Lemma \ref{detectGacy2}.  Our proof  of Theorem \ref{rationalres} follows from $(\dag_6)$,
Lemma \ref{starfacts}(5,6) (remember that $n\geq2$ in this setting), and Lemma \ref{detectGacy1}.
\end{proof}

\end{document}